\documentclass[12pt,a4paper]{amsart}

\usepackage{fullpage,amsxtra,amssymb,eucal,amsthm,amsmath,url,listings}
\usepackage[utf8]{inputenc}
\usepackage[curve]{xypic}

\numberwithin{equation}{section}
\renewcommand{\leq}{\leqslant}
\renewcommand{\geq}{\geqslant}

\def\stacksum#1#2{{\stackrel{{\scriptstyle #1}}
{{\scriptstyle #2}}}}

\newcommand{\Cc}{\mathbf{C}}

\newcommand{\Zz}{\mathbf{Z}}

\newcommand{\Bb}{\mathbf{B}}
\newcommand{\Uu}{\mathbf{U}}

\newcommand{\Gg}{\mathbf{G}}
\newcommand{\Tt}{\mathbf{T}}

\newcommand{\Fp}{\mathbf{F}}
\newcommand{\Ss}{\mathbf{S}}

\newcommand{\mods}[1]{\,(\mathrm{mod}\,{#1})}

\DeclareMathSymbol{\dgenb}{\mathord}{symbols}{"1F}
\DeclareMathSymbol{\dgena}{\mathord}{symbols}{"1E}
\DeclareMathSymbol{\dgenaa}{\mathord}{symbols}{"1C}
\DeclareMathSymbol{\dgenbb}{\mathord}{symbols}{"1D}

\newcommand{\ra}{\rightarrow}
\newcommand{\lra}{\longrightarrow}

\newcommand{\eps}{\varepsilon}
\renewcommand{\rho}{\varrho}

\newcommand{\demi}{{\textstyle{\frac{1}{2}}}}
\newcommand{\kesten}{{\textstyle{\frac{2}{\sqrt{3}}}}}

\newcommand{\girth}[1]{\mathrm{girth({#1})}}
\newcommand{\cayley}[2]{\mathcal{C}({#1},{#2})}

\newcommand{\vois}[2]{\mathcal{B}_{{#1}}({#2})}

\newcommand{\mdim}[1]{d({{#1}})}

\newcommand{\bbigf}{\mathcal{G}}
\newcommand{\rp}[1]{\mathrm{rp}({#1})}
\newcommand{\rpp}[1]{\mathrm{rp}^+({#1})}

\newcommand{\app}[1]{\mathsf{{#1}}}
\newcommand{\trp}[1]{\mathrm{trp}({#1})}
\newcommand{\nfold}[2]{{{#1}}^{{(#2)}}}

\newcommand{\rd}[1]{\ar@{{*}-{*}}[r]_-{{#1}}}
\newcommand{\lrd}[1]{\ar@{{*}-{*}}[rr]_-{{#1}}}
\newcommand{\crd}[1]{\ar @/_1pc/ @{{*}-{*}}[r]_-{{#1}}}
\newcommand{\lcrd}[1]{\ar @/_1pc/ @{{*}-{*}}[rr]_-{{#1}}}
\newcommand{\fb}[1]{\ar@{{*}->}[r]^-{{#1}}}
\newcommand{\lfb}[1]{\ar@{{*}->}[rr]^-{{#1}}}
\newcommand{\me}[1]{\ar@{{*}~{*}}[r]^-{{#1}}}
\newcommand{\sub}{\ \ar@{>->}[r]}

\newcommand{\uple}[1]{\text{\boldmath${#1}$}}

\newcommand{\lmap}[4]{\left\{\begin{array}{ccl}{{#1}}&\lra&{{#2}}\\{{#3}}&\mapsto&{{#4}}\end{array}\right.}

\DeclareMathSymbol{\gena}{\mathord}{letters}{"3C}
\DeclareMathSymbol{\genb}{\mathord}{letters}{"3E}

\DeclareMathOperator{\cl}{Cl}
\DeclareMathOperator{\bcl}{\mathbf{Cl}}

\DeclareMathOperator{\Tr}{Tr}

\DeclareMathOperator{\diam}{diam}

\DeclareMathOperator{\proba}{\mathbf{P}}

\DeclareMathOperator{\SL}{SL}
\DeclareMathOperator{\GL}{GL}

\DeclareMathOperator{\PSL}{PSL}

\theoremstyle{plain}
\newtheorem{theorem}{Theorem}[section]

\newtheorem{lemma}[theorem]{Lemma}
\newtheorem{corollary}[theorem]{Corollary}

\newtheorem{proposition}[theorem]{Proposition}

\theoremstyle{remark}
\newtheorem{remark}[theorem]{Remark}

\theoremstyle{definition}

\newtheorem{definition}[theorem]{Definition}

\newcounter{exercices}
\newcounter{etape}

\begin{document}

\title[Explicit expansion]{Explicit growth and expansion for $\SL_2$}
 
\author{Emmanuel  Kowalski}
\address{ETH Z\"urich -- D-MATH\\
  R\"amistrasse 101\\
  8092 Z\"urich\\
  Switzerland} 
\email{kowalski@math.ethz.ch}

\date{}

\subjclass[2010]{20F69, 05C50, 05C81} 

\keywords{Growth of finite groups, expander graphs, Cayley graphs,
  diameter, random walks on groups}

\begin{abstract}
We give explicit versions of Helfgott's Growth Theorem for $\SL_2$, as
well as of the Bourgain-Gamburd argument for expansion of Cayley
graphs modulo primes of  subgroups of $\SL_2(\Zz)$ which are
Zariski-dense in $\SL_2$. 
\end{abstract}

\maketitle

\tableofcontents

\section{Introduction}

Our main goal in this paper is to prove the following result, which is
an explicit version of a theorem of Bourgain and Gamburd~\cite{bg}:

\begin{theorem}\label{th-bg-bis}
  Let $S\subset \SL_2(\Zz)$ be a finite symmetric set such that the
  subgroup generated by $S$ is Zariski-dense in $\SL_2(\Zz)$. Let
  $\mathcal{P}$ be the set of primes such that $S_p=S\mods{p}$
  generates $\SL_2(\Fp_p)$, which contains all but finitely many
  primes. Then the family of Cayley graphs
  $(\cayley{\SL_2(\Fp_p)}{S_p})_{p\in\mathcal{P}}$ is an expander
  family, and one can write down explicit bounds for the spectral gap,
  given the set $S$.
\par
In particular, if $S$ generates a free group of rank $|S|/2$, the
spectral gap\footnote{\ This is the spectral gap of the
  \emph{normalized} Laplace operator $\Delta=\mathrm{Id}-M$, where $M$
  is the Markov averaging operator of the graph; thus the spectrum of
  $\Delta$ is a subset of the interval $[0,2]$.} satisfies
\begin{equation}\label{eq-stated-sg}
  \lambda_1(\cayley{\SL_2(\Fp_p)}{S_p})\geq
  2^{-2^{35}\gamma^{-1}}
\end{equation}
for all $p$ large enough, where
$$
\gamma=\frac{\log(\kesten\sqrt{ |S|})}{\log \max_{s\in S}{\|s\|}},
$$
the norm $\|s\|$ being the operator norm of the matrix $s$, with
respect to the euclidean metric on $\Cc^2$. 
\end{theorem}

We can specify what ``$p$ large enough'' means, but we defer a
statement to Section~\ref{ssec-summary} since this involves a series
of inequalities which are awkward to state (and unenlightening), but
easy to check for a given concrete set of matrices $S$.
\par
A crucial ingredient in the argument of Bourgain and Gamburd is
Helfgott's Growth Theorem~\cite{helfgott} for $\SL_2$, which has
considerable independent interest. We thus require an explicit version
of it, and we will prove the following:

\begin{theorem}\label{th-helfgott}
  Let $p$ be a prime number, $H\subset \SL_2(\Fp_p)$ a symmetric
  generating subset of $\SL_2(\Fp_p)$ containing $1$. Then the triple
  product set $\nfold{H}{3}=H\cdot H\cdot H$ satisfies either
  $\nfold{H}{3}=\SL_2(\Fp_p)$ or
$$
|\nfold{H}{3}|\geq |H|^{1+\delta},
$$
where $\delta=1/3024$.
\end{theorem}

Here is a simple corollary, which is (as far as the author is aware)
also the first explicit result of this kind for almost simple linear
groups:

\begin{corollary}[Explicit solution to Babai's conjecture for
  $\SL_2(\Fp_p)$]\label{cor-babai}
For any prime number $p$ and any symmetric generating set
  $S$ of $\SL_2(\Fp_p)$, we have
$$
\diam \cayley{\SL_2(\Fp_p)}{S}\leq 3(\log |\SL_2(\Fp_p)|)^{C}
$$
with $C=3323$.
\end{corollary}

Another corollary of Helfgott's Theorem and of intermediate results
used in the proof of Theorem~\ref{th-bg-bis} is a better diameter
bound for Zariski-dense subgroups:

\begin{corollary}[Diameter bounds for Zariski-dense subgroups of $\SL_2$]\label{cor-diameter}
  Let $S\subset \SL_2(\Zz)$ be a finite symmetric set such that the
  subgroup generated by $S$ is Zariski-dense in $\SL_2(\Zz)$ and is a
  free group of rank $|S|/2$. Let $\mathcal{P}$ be the set of primes
  such that $S_p=S\mods{p}$ generates $\SL_2(\Fp_p)$. 
\par
Let $\delta>0$ be as in Helfgott's Theorem and define
$$
\tau^{-1}=\log \max_{s\in S}{\|s\|}>0.
$$
\par
Then for $p\in\mathcal{P}$ and $p>\exp(2/\tau)$, we have
$$
\diam(\cayley{\SL_2(\Fp_p)}{S})\leq 3^A(\log |\SL_2(\Fp_p)|)
$$
where
$$
A=\frac{\log (8\tau^{-1}(|S|-1)^{-1})}{\log(1+\delta)}.
$$
\end{corollary}

\begin{remark}
  Using the well-known bound
$$
\lambda_1(\Gamma_p)\geq \frac{1}{|S|\diam(\Gamma_p)^2}
$$
(see, e.g., ~\cite[Th. 13.23]{lpw}), these diameter bounds can be used
to get lower bounds for spectral gaps for ``medium'' primes. Note the
huge discrepancy however at the end of the range.
\end{remark}

Combining Theorem~\ref{th-bg-bis} with the second corollary, we can
give explicit statements for the motivating example of the Lubotzky
group.

\begin{corollary}[The Lubotzky group]\label{cor-lubotzky}
  Let
$$
S=\Bigl\{\begin{pmatrix}1&\pm 3\\0&1
\end{pmatrix},
\begin{pmatrix}1&0\\\pm 3&1
\end{pmatrix}
\Bigr\}\subset \SL_2(\Zz),
$$
and let $\Gamma_p=\cayley{\SL_2(\Fp_p)}{S_p}$.  Then we have
\begin{equation}\label{eq-sg-lubotzky}
\lambda_1(\Gamma_p)\geq 2^{-2^{36}}
\end{equation}
if $p\geq 2^{2^{46}}$, and 
$$
\diam(\Gamma_p)\leq 2^{5572}(\log |\SL_2(\Fp_p)|)
$$
for all $p\not=3$. 
\end{corollary}


The original papers of Bourgain and Gamburd~\cite{bg} and
Helfgott~\cite{helfgott} are effective, and thus it is not surprising
that one can obtain explicit versions. What is less clear is how good
the constants may be, and how much work may be required to provide
them. This paper gives a first indication in that respect. 
\par
The bounds we derive are very unlikely to be anywhere near sharp, and
not only because we often use rather rough estimates to simplify the
shape and constants appearing in various inequalities.\footnote{\ In
  some cases, one can easily extract better bounds from the proof,
  e.g., one can replace $1/3024$ by $1/1513$ for all $H$ large enough
  in Theorem~\ref{th-helfgott}.}  Indeed, when the Hausdorff dimension
of the limit set of the subgroup $G$ generated by $S$ is large enough,
Gamburd~\cite{gamburd} has shown quite good spectral gaps for the
hyperbolic Laplace operator on $G\backslash \mathbf{H}$, which
strongly suggest that the corresponding combinatorial spectral gap
would be also relatively large. But this computation has not been
done, to the author's knowledge, and our version of
Theorem~\ref{th-bg-bis} gives the first fully explicit spectral gap
for infinite-index subgroups of $\SL_2(\Zz)$, with
Corollary~\ref{cor-lubotzky} being a nice concrete example (it is also
known that the ``Lubotzky group'' is too small for Gamburd's result to
apply).
\par
In view of the direct link between the spectral gap of families of
Cayley graphs of quotients of ``thin'' (or sparse) subgroups of
arithmetic groups and quantitative applications of sieve methods to
these groups, it is natural to wish for a better understanding of
these issues.\footnote{\ Indeed, this question was asked by J-P. Serre
  during the author's Bourbaki lecture~\cite{bourbaki}.} A first step
towards effective versions of these applications of ``sieve in orbit''
would be to extend Theorem~\ref{th-bg-bis} to an effective spectral
gap for $\SL_2(\Zz/q\Zz)$, where $q$ is a squarefree modulus (as
originally proved by Bourgain, Gamburd and Sarnak~\cite{bgs}), and we
hope to come back to this.
\par
As a final remark, the reader can also see this paper as presenting a
complete proof of the qualitative forms of Theorems~\ref{th-bg-bis}
and~\ref{th-helfgott} and their corollaries. When read in this light,
ignoring the fussy technical details arising from trying to have
explicit bounds, it may in fact be useful as a self-contained
introduction to this area of research.
\par
\medskip 
\par
\textbf{Notation.} As usual, $|X|$ denotes the cardinality of a set.
Given a group $G$, and a symmetric generating set $S$, we denote by
$\cayley{G}{S}$ the Cayley graph of $G$ with respect to $S$, which is
$|S|$-regular. Moreover, we say that a symmetric set $S\subset G$
\emph{freely generates} $G$ if representatives of $S$ modulo the
relation $s\sim s^{-1}$ form a free generating set of $G$, i.e., $G$
is a free group of rank $|S|/2$.
\par
For a subset $H\subset G$ of a group $G$, we write $\nfold{H}{n}$ for
the $n$-fold product set
$$
\nfold{H}{n}=\{x\in G\,\mid\, x=h_1\cdots h_n\text{ for some } h_i\in
H\}.
$$
\par
Note the immediate relations
$$
\nfold{(\nfold{H}{n})}{m}=\nfold{H}{nm},\quad
\nfold{H}{n+m}=\nfold{H}{n}\cdot \nfold{H}{m}
$$
for $n$, $m\geq 0$ and $(\nfold{H}{n})^{-1}=\nfold{H}{n}$ if $H$ is
symmetric. In addition, if $1\in H$, we have $\nfold{H}{n}\subset
\nfold{H}{m}$ for all $m\geq n$. In particular, the diameter of a
Cayley graph $\cayley{G}{H}$, when $H=H^{-1}$, is the smallest $n\geq
1$ such that $\nfold{\tilde{H}}{n}=G$, where $\tilde{H}=H\cup \{1\}$.
\par
We denote by $\trp{H}$ the ``tripling constant'' of a subset $H\subset
G$, defined by
$$
\trp{H}=\frac{|\nfold{H}{3}|}{|H|}.
$$
\par
\medskip 
\par
\textbf{Acknowledgments.} Much of the work on this paper was done
during and following a course on expander graphs that I taught at ETH
Zürich during the Fall Semester 2011.\footnote{\ Lecture notes for
  this course are available, and contain more background and
  motivating material~\cite{expanders}.} Thanks to all who attended
this course and helped with corrections and remarks. Thanks in
particular to O. Dinai, and to R. Pink for very interesting
discussions and for helping with the proof of the specific variant of
a ``Larsen-Pink'' inequality in Theorem~\ref{th-lp}. Thanks very much
to P. Sarnak for his comments, and especially for his insights
concerning the history of questions and results concerning spectral
gaps for subgroups of $\SL_2(\Zz)$. Thanks to L. Pyber for clarifying
some of the ``combinatorics'' in the proof of the growth theorem.
\par
Finally, I owe a great debt to the referee for his or her extremely
detailed reading of the paper, which found many computational mistakes
in the first draft. In a paper of this kind, this makes of course an
enormous difference.


\section{Explicit multiplicative combinatorics}

Another ingredient of Theorem~\ref{th-bg-bis} is the relation between
subsets of a finite group with small ``multiplicative energy'' and
sets with small tripling constant, or approximate subgroups. This was
obtained by Tao~\cite{tao}, in good qualitative form, but without
explicit dependency of the various quantities involved. In this
section, we state a suitably explicit version.
\par
We recall first the definitions involved. For a finite group $G$ and
$A$, $B\subset G$, one defines the multiplicative energy by
$$
E(A,B)=|\{(a_1,a_2,b_1,b_2)\in A^2\times B^2\,\mid\, a_1b_1=a_2b_2\}|.
$$
\par
It is also convenient to denote by
$$
e(A,B)=\frac{|E(A,B)|}{(|A||B|)^{3/2}}.
$$
the normalized multiplicative energy, which is $\leq 1$.  Following
Tao (see~\cite[Def. 3.8]{tao}), for a finite group $G$ and any
$\alpha\geq 1$, a subset $\app{H}\subset G$ is an $\alpha$-approximate
subgroup if $1\in \app{H}$, $\app{H}=\app{H}^{-1}$ and there exists a
symmetric subset $X\subset \nfold{H}{2}$ of order at most $\alpha$
such that
\begin{equation}\label{eq-def-app}
  \app{H}\cdot \app{H}\subset X\cdot \app{H},
\end{equation}
which implies also $\app{H}\cdot \app{H}\subset \app{H}\cdot X$. Then
we have:

\begin{theorem}\label{th-large-energy}
  Let $G$ be a finite group and $\alpha\geq 1$. If $A$ and $B$ are
  subsets of $G$ such that $e(A,B)\geq \alpha^{-1}$, there exist
  constants $\beta_1$, $\beta_2$, $\beta_3$, $\beta_4\geq 1$, a
  $\beta_1$-approximate subgroup $\app{H}\subset G$ and elements $x$,
  $y\in G$ such that
\begin{gather*}
  |\app{H}|\leq \beta_2|A|\leq \beta_2\alpha^2|B|,\\
  |A\cap x\app{H}|\geq \frac{1}{\beta_3}|A|,\quad\quad |B\cap
  \app{H}y|\geq \frac{1}{\beta_3}|B|,\\
  \trp{\app{H}}\leq \beta_4,
\end{gather*}
and moreover $\beta_i\leq c_1\alpha^{c_2}$ for some \emph{absolute}
constants $c_1$, $c_2>0$. In fact, one can take
\begin{equation}\label{eq-explicit-mc}
  \beta_1\leq 2^{1861}\alpha^{720},\quad\quad \beta_2\leq
  2^{325}\alpha^{126},\quad\quad \beta_3\leq
  2^{2424}\alpha^{937},\quad\quad \beta_4\leq 2^{930}\alpha^{360}.
\end{equation}
\end{theorem}

Except for the values of the constants, this is proved
in~\cite[Th. 5.4, (i) implies (iv)]{tao} and quoted
in~\cite[Th. 2.48]{tao-vu}. Since this is obtained by following line
by line the arguments of Tao, we defer a proof to the Appendix.

\section{Growth for $\SL_2$}

We prove here Theorem~\ref{th-helfgott}. The argument we use is
basically the one sketched by Pyber and Szab\'o in~\cite[\S 1.1]{psz}
(which is expanded in their paper to cover much more general
situations). It is closely related to the one of Breuillard, Green and
Tao~\cite{bgt}, and many ingredients are already visible in Helfgott's
original argument~\cite{helfgott}.

\subsection{Elementary facts and definitions}

We begin with an important observation, which applies to all finite
groups, and goes back to Ruzsa: to prove that the tripling constant of
a generating set $H$ is at least a small power of $|H|$, it is enough
to prove that the growth ratio after an arbitrary (but fixed) number
of products is of such order of magnitude.

\begin{proposition}[Ruzsa]\label{pr-ruzsa}
  Let $G$ be a finite group, and let $H\subset G$ be a symmetric
  non-empty subset.
\par
\emph{(1)} Denoting $\alpha_n=|\nfold{H}{n}|/|H|$, we have
\begin{equation}\label{eq-ruzsa}
\alpha_{n}\leq \alpha_3^{n-2}=\trp{H}^{n-2}
\end{equation}
for all $n\geq 3$.
\par
\emph{(2)} We have $\trp{\nfold{H}{2}}\leq \trp{H}^4$ and for $k\geq
3$, we have
$$
\trp{\nfold{H}{k}}\leq \trp{H}^{3k-3}.
$$
\end{proposition}

\begin{proof}
  The first part is well-known (see, e.g.,~\cite[Proof of Lemma
  2.2]{helfgott}). For (2), we have
$$
\trp{\nfold{H}{k}}=\frac{\alpha_{3k}}{\alpha_k}.
$$
\par
Since $\alpha_{k}\geq \alpha_3$ for $k\geq 3$, we obtain
$\trp{\nfold{H}{k}}\leq \alpha_3^{3k-3}$ for $k\geq 3$ by (1), while
for $k\geq 2$, we simply use $\alpha_2\geq 1$ to get
$\trp{\nfold{H}{2}}\leq \alpha_3^4$.
\end{proof}

We first use Ruzsa's Lemma to show that Helfgott's Theorem holds when
$|H|$ is small, in the following sense:

\begin{lemma}\label{lm-small-p}
  Let $G$ be a finite group and let $H$ be a symmetric generating set
  of $G$ containing $1$. If $\nfold{H}{3}\not=G$, we have
  $|\nfold{H}{3}|\geq 2^{1/2}|H|$.
\end{lemma}

\begin{proof}
  If the triple product set is not all of $G$, it follows that
  $\nfold{H}{3}\not=\nfold{H}{2}$. We fix some $x\in
  \nfold{H}{3}-\nfold{H}{2}$, and consider the injective map
$$
i\,:\, \lmap{H}{G}{h}{hx}.
$$
\par
The image of this map is contained in $\nfold{H}{4}$ and it is
disjoint with $H$ since $x\notin\nfold{H}{2}$. Hence $\nfold{H}{4}$,
which contains $H$ and the image of $i$, satisfies $|\nfold{H}{4}|\geq
2|H|$. Hence, by Ruzsa's Lemma, we obtain
$$
\trp{H}\geq \Bigl(\frac{|\nfold{H}{4}|}{|H|}\Bigr)^{1/2}\geq 2^{1/2}. 
$$
\end{proof}

\begin{remark}
  In fact, as the referee pointed out, a better result is known (and
  is elementary): if $H\subset G$ generates $G$ then $\nfold{2}{H}\geq
  \frac{3}{2}|H|$ (see~\cite{tao-blog}).
\end{remark}

The following version of the orbit-stabilizer theorem will be used to
reduce the proof of lower-bounds on the size a set to an upper-bound
for another.

\begin{proposition}[Helfgott]
  Let $G$ be a finite group acting on a non-empty finite set $X$. Fix
  some $x\in X$ and let $K\subset G$ be the stabilizer of $x$ in
  $G$. For any non-empty symmetric subset $H\subset G$, we have
$$
|K\cap \nfold{H}{2}|\geq \frac{|H|}{|H\cdot x|}
$$
where $H\cdot x=\{h\cdot x\,\mid\, h\in H\}$.
\end{proposition}

(Note that since $H$ is symmetric, we have $1\in K\cap \nfold{H}{2}$.)

\begin{proof}
  As in the classical proof of the orbit-stabilizer theorem, we
  consider the orbit map, but restricted to $H$
$$
\phi\,:\,
\lmap{H}{X}{h}{h\cdot x}.
$$
\par
Using the fibers of this map to count the number of elements in $H$,
we get
$$
|H|=\sum_{y\in \phi(H)}{|\phi^{-1}(y)|}.
$$
\par
But the image of $\phi$ is $\phi(H)=H\cdot x$, and we have
$$
|\phi^{-1}(y)|\leq |K\cap \nfold{H}{2}|
$$
for all $y$ (indeed, if $y=\phi(h_0)$ with $h_0\in H$, then all
elements $h\in H$ with $\phi(h)=y$ satisfy $h_0^{-1}h\in K\cap
\nfold{H}{2}$). Therefore we get
$$
|H|\leq |H\cdot x||K\cap \nfold{H}{2}|,
$$
as claimed.
\end{proof}

Finally, a last lemma shows that if a subset $H$ has small tripling
constant ``in a subgroup'', then $H$ itself has small tripling (in the
language of approximate groups, it is a special case of the fact that
the intersection of two approximate groups is still one).

\begin{lemma}\label{lm-intersection}
Let $G$ be a finite group, $K\subset G$ a subgroup, and $H\subset G$
an arbitrary symmetric subset.  For any $n\geq 1$, we have
$$
\frac{|\nfold{H}{n+1}|}{|H|}
\geq \frac{|\nfold{H}{n}\cap K|}{|\nfold{H}{2}\cap K|}.
$$
\end{lemma}

\begin{proof}
Let $X\subset G/K$ be the set of cosets of $K$ intersecting $H$:
$$
X=\{xK\in G/K\,\mid\, xK\cap H\not=\emptyset\}.
$$
\par
We can estimate the size of this set from below by splitting $H$ into
its intersections with cosets of $K$: we have
$$
|H|=\sum_{xK\in X}{|H\cap xK|}.
$$
\par
But for any $xK\in X$, fixing some $g_0\in xK\cap H$, we have
$g^{-1}g_0\in K\cap \nfold{H}{2}$ if $g\in xK\cap H$, and hence
$$
|xK\cap H|\leq |K\cap\nfold{H}{2}|.
$$
\par
This gives the lower bound
$$
|X|\geq \frac{|H|}{|K\cap \nfold{H}{2}|}.
$$
\par
Now take once more some $xK\in X$, and fix an element $xk=h\in xK\cap
H$. Then all the elements $xkg$ are distinct for $g\in K$, and they
are in $xK\cap \nfold{H}{n+1}$ if $g\in K\cap \nfold{H}{n}$, so that
$$
|xK\cap \nfold{H}{n+1}|\geq |K\cap \nfold{H}{n}|
$$
for any $xK\in X$, and (cosets being disjoint)
$$
|\nfold{H}{n+1}|\geq |X||K\cap \nfold{H}{n}|,
$$
which gives the result when combined with the lower bound for $|X|$.
\end{proof}

We will use classical structural definitions and facts about finite
groups of Lie type. In particular, a regular semisimple element $g\in
\Gg=\SL_2(\bar{\Fp}_p)$ is a semisimple element with distinct
eigenvalues. The centralizer of such an element is a maximal torus in
$\Gg$. For any subset $H\subset \Gg$, we write $H_{reg}$ for the set
of the regular semisimple elements in $H$. A maximal torus $T\subset
G=\SL_2(\Fp_p)$ is the intersection $G\cap \Tt$, where $\Tt$ is a
maximal torus of $\Gg$ which is stable under the Frobenius
automorphism $\sigma$. Here are the basic properties of regular
semisimple elements and their centralizers; these are all standard
facts, and we omit the proofs. (For general facts about finite groups
of Lie type, one may look at~\cite{digne-michel} or~\cite[Ch. 1 and
3]{carter}, and for conjugacy classes in $\SL_2(\Fp_p)$, one may look
for instance at~\cite[p. 70]{fulton-harris}; another source for
$\SL_2$ is~\cite[Ch. 6]{suzuki}).

\begin{proposition}\label{pr-tori-elementary}
  Fix a prime number $p$ and let $G=\SL_2(\Fp_p)$,
  $\Gg=\SL_2(\bar{\Fp}_p)$.
\par
\emph{(1)} A regular semisimple element $x\in \Gg$ is contained in a
\emph{unique} maximal torus $\Tt$, namely its centralizer
$\Tt=C_{\Gg}(x)$. In particular, if $\Tt_1\not=\Tt_2$ are two maximal
tori, we have
\begin{equation}\label{eq-dist-tori-inter}
\Tt_{1,reg}\cap \Tt_{2,reg}=\emptyset.
\end{equation}
\par
\emph{(2)} If $\Tt\subset \Gg$ is a maximal torus, we have
$$
|\Tt_{nreg}|=|\Tt-\Tt_{reg}|=2.
$$
\par
\emph{(3)} For any maximal torus $\Tt$, the normalizer $N_{\Gg}(\Tt)$
contains $\Tt$ as a subgroup of index $2$. Similarly, for any maximal
torus $T\subset G$, $N_G(T)$ contains $T$ as a subgroup of index
$2$, and in particular
$$
2(p-1)\leq |N_G(T)|\leq 2(p+1).
$$
\par
\emph{(4)} The conjugacy class $\bcl(g)$ of a regular semisimple
element $g\in \Gg$ is the set of all $x\in \Gg$ such that
$\Tr(x)=\Tr(g)$. The set of elements in $\Gg$ which are not regular
semisimple is the set of all $x\in \Gg$ such that $\Tr(x)^2=4$.
\end{proposition}

Finally, (a variant of) the following concept was introduced under
different names and guises by Helfgott, Pyber-Szab\'o, and
Breuillard-Green-Tao. We chose the name from the last team.

\begin{definition}[A set involved with a torus]
  Let $p$ be a prime number, $H\subset \SL_2(\Fp_p)$ a finite set and
  $\Tt\subset \SL_2(\bar{\Fp}_p)$ a maximal torus. Then \emph{$H$ is
    involved with $\Tt$}, or $\Tt$ with $H$, if and only if $\Tt$ is
  $\sigma$-invariant and $H$ contains a regular semisimple element of
  $\Tt$ with non-zero trace, i.e., $H\cap \Tt_{sreg}\not=\emptyset$
  where the superscript ``sreg'' restricts to regular semisimple
  elements with non-zero trace.
\end{definition}

\begin{remark}
  The twist in this definition, compared with the one in~\cite{psz}
  or~\cite{bgt}, is that we insist on having non-zero trace. This will
  be helpful later on, as it will eliminate a whole subcase in the key
  estimate (the proof of Proposition~\ref{pr-dichotomy}), and lead to
  a shorter proof, with better explicit constants. However, this
  restriction is not really essential in the greater scheme of things,
  and it would probably not be a good idea to do something similar for
  more general groups.
\end{remark}

The alternative $\nfold{H}{3}=\SL_2(\Fp_p)$ in Helfgott's growth
theorem will be obtained as a corollary of the Gowers-Nikolov-Pyber
``quasi-random groups'' argument (see~\cite{gowers} and~\cite{np}).

\begin{proposition}\label{pr-qr}
  For a prime $p\geq 3$, if a subset $H\subset \SL_2(\Fp_p)$ satisfies
 $$
|H|\geq 2|\SL_2(\Fp_p)|^{8/9},
$$
we have $\nfold{H}{3}=\SL_2(\Fp_p)$.
\end{proposition}

For a proof, see, e.g.,~\cite[\S 4.5]{expanders}.

\subsection{Escape from subvarieties and non-concentration lemmas}

Two important tools in the proof of growth theorems for linear groups
are estimates for escape from subvarieties and for non-concentration
in subvarieties. We state and prove in this section the special cases
which we need for the explicit proof of Helfgott's Theorem. The reader
may wish to look only at the statements and skip afterwards to the
next section to see how they are used.

\begin{lemma}[Escape]\label{lm-escape}
  Let $p\geq 7$ be a prime number and let $H\subset \SL_2(\Fp_p)$ be a
  symmetric generating set with $1\in H$. Then
  $\nfold{H}{3}_{sreg}\not=\emptyset$, i.e., the three-fold product
  set $\nfold{H}{3}$ contains a regular semisimple element $x$ with
  non-zero trace.\footnote{\ The condition $p\geq 7$ is sharp,
    see~\cite[Example 4.6.13]{expanders} for an example.} In
  particular, there exists a torus $\Tt=\Cc_{\Gg}(x)$ involved with
  $\nfold{H}{3}$.
\end{lemma}

The general non-concentration inequalities are now often called
``Larsen-Pink inequalities'', since the first versions appeared in the
work of Larsen and Pink~\cite{larsen-pink} on finite subgroups of
linear groups. ``Approximate'' versions occur in the work of
Hrushovski~\cite{hrushovski} and Breuillard-Green-Tao~\cite{bgt}, with
closely related results found in that of Pyber and Szab\'o~\cite{psz}.

\begin{theorem}[Non-concentration inequality]\label{th-lp}
  Let $p\geq 3$ be a prime number and let $g\in \SL_2(\Fp_p)=G$ be a
  regular semisimple element with non-zero trace. Let $\bcl(g)\subset
  \SL_2(\bar{\Fp}_p)=\Gg$ be the conjugacy class of $g$. If $H\subset
  G$ is a symmetric generating set containing $1$, we have
\begin{equation}\label{eq-lp1}
  |\bcl(g)\cap H|\leq 7\alpha^{2/3}|H|^{2/3}
\end{equation}
where $\alpha=\trp{H}$ is the tripling constant of $H$, \emph{unless}
\begin{equation}\label{eq-lp2}
  \alpha>|H|^{1/28}.
\end{equation}
\end{theorem}

From this last fact, we will deduce the following dichotomy, which is
the precise tool used in the next section to prove Helfgott's
Theorem. 

\begin{proposition}[Involving dichotomy]\label{pr-dichotomy}
\par
\emph{(1)} For all prime number $p$, all subsets $H\subset
\SL_2(\Fp_p)$ and all maximal tori $\Tt\subset \SL_2(\bar{\Fp}_p)$, if
$\Tt$ and $H$ are not involved, we have
$$
|H\cap \Tt|\leq 4.
$$
\par
\emph{(2)} If $p\geq 3$ and $H\subset \SL_2(\Fp_p)=G$ is a symmetric
generating set containing $1$, we have
\begin{equation}\label{eq-key-ineq}
  |\Tt_{reg}\cap \nfold{H}{2}|\geq 14^{-1}\alpha^{-4}|H|^{1/3}
\end{equation}
for any maximal torus $\Tt\subset \SL_2(\bar{\Fp}_p)$ which is
involved with $H$, where $\alpha=\trp{H}$, unless
\begin{equation}\label{eq-key-altern}
  \alpha\geq |H|^{1/168}.
\end{equation}
\end{proposition}

\begin{proof}
  (1) is obvious, since $|\Tt-\Tt_{reg}|\leq 2$ and there are also at
  most two elements of trace $0$ in $\Tt$ (as one can check
  quickly).
\par
For (2), we apply the orbit-stabilizer theorem. Let $T=\Tt\cap \Gg$ be
a maximal torus in $G$. Fixing any $g\in T_{reg}$, we have
$T=C_{G}(g)$, the stabilizer of $g$ in $G$ for its conjugacy action on
itself. We find that
\begin{equation}\label{eq-appl-orbit-stab}
|\Tt\cap \nfold{H}{2}|\geq \frac{|H|}{|\{hgh^{-1}\,\mid\, h\in H\}|}
\end{equation}
for any symmetric subset $H$. Since $H$ is involved with $\Tt$, we can
select $g$ in $T_{sreg}\cap H$ in this inequality, and the denominator
on the right becomes
$$
|\{hgh^{-1}\,\mid\, h\in H\}|\leq |\nfold{H}{3}\cap \cl(g)|\leq
|\nfold{H}{3}\cap \bcl(g)|
$$
where $\cl(g)$ is the conjugacy class of $g$ in $G$. Applying the
Larsen-Pink inequality to $\nfold{H}{3}$, with tripling constant
bounded by $\alpha^6$ (by Ruzsa's Lemma), we obtain the lower bound
$$
|\Tt\cap \nfold{H}{2}|\geq \frac{|H|}{
|\nfold{H}{3}\cap \bcl(g)|}
\geq 
7^{-1}\alpha^{-4}|H|^{1/3},
$$
unless $\alpha=\trp{H}\geq |H|^{1/168}$. In the first case, we get
$$
|\Tt_{reg}\cap \nfold{H}{2}|\geq 14^{-1}\alpha^{-4}|H|^{1/3},
$$
unless
$$
7^{-1}\alpha^{-4}|H|^{1/3}\leq 2
$$
since there are only two elements of $\Tt\cap \nfold{H}{2}$ which are
not regular. This last alternative gives
$$
\alpha\geq \demi |H|^{1/12}
$$
which we see is a stronger conclusion than~(\ref{eq-key-altern})
(precisely, it is strictly stronger if $|H|>2^{13}$, but in the other
case the lower bound $\trp{H}\geq \sqrt{2}$ from
Lemma~\ref{lm-small-p} is already a better result.)  Hence
Proposition~\ref{pr-dichotomy} is proved.
\end{proof}

Now we prove the escape and non-concentration results.

\begin{proof}[Proof of Lemma~\ref{lm-escape}]
  The basic point that allows us to give a quick proof is that the set
  $N\subset \SL_2(\Fp_p)$ of elements which are not regular semisimple
  is invariant under $\SL_2(\Fp_p)$-conjugation, and is the set of all
  matrices with trace equal to $2$ or $-2$. It is precisely the union
  of the two central elements $\pm 1$ and the four conjugacy classes
  of
$$
u=\begin{pmatrix}1&1\\0&1
\end{pmatrix},\quad v=\begin{pmatrix}-1&1\\0&-1
\end{pmatrix},\quad
u'=\begin{pmatrix}1&\eps\\0&1
\end{pmatrix},\quad v'=\begin{pmatrix}-1&\eps\\0&-1
\end{pmatrix}
$$
(where $\eps\in\Fp_p^{\times}$ is a fixed non-square) while elements
of trace $0$ are the conjugates of
$$
g_0=\begin{pmatrix}0&1\\-1&0
\end{pmatrix}
$$
(these are standard facts, which can be checked on the list of
conjugacy classes in~\cite[p. 70]{fulton-harris}, for instance.)
\par
We next note that, if the statement of the lemma fails for a given
$H$, it also fails for every conjugates of $H$, and that this allows
us to normalize at least one element to a specific representative of
its conjugacy class. It is convenient to argue by contradiction,
though this is somewhat cosmetic. So we assume that
$\nfold{H}{3}_{nreg}$ is empty and $p\geq 7$, and will derive a
contradiction.
\par
We distinguish two cases. In the first case, we assume that $H$
contains one element of trace $\pm 2$ which is not $\pm 1$. The
observation above shows that we can assume that one of $u$, $v$, $u'$,
$v'$ is in $H$. We deal first with the case $u\in H$.
\par
Since $H$ is a symmetric generating set, it must contain some element
$$
g=\begin{pmatrix}a&b\\c&d
\end{pmatrix},
$$
with $c\not=0$, since otherwise, all elements of $H$ would be
upper-triangular, and $H$ would not generate $\SL_2(\Fp_p)$. Then
$\nfold{H}{3}$ contains $ug$, $u^2g$, $u^{-1}g$, $u^{-2}g$, which have
traces, respectively, equal to $\Tr(g)+c$, $\Tr(g)+2c$, $\Tr(g)-c$,
$\Tr(g)-2c$. Since $c\not=0$, and $p$ is not $2$ or $3$, we see that
these traces are distinct, and since there are $4$ of them, one at
least is not in $\{-2,0,2\}$, which contradicts our assumption.
\par
If $u'\in H$, the argument is almost identical. If $u'$ (or similarly
$v'$) is in $H$, the set of traces of $(u')^jg$ for $j\in
\{-2,-1,0,1,2\}$ is
$$
\{\Tr(g)+2c,-\Tr(g)-c,\Tr(g),-\Tr(g)+c,\Tr(g)-2c\},
$$
and one can check that for $p\geq 5$, one of these is not $0$, $-2$ or
$2$, although some could coincide (for instance, if $\Tr(g)=2$, the
other traces are $\{2+2c,-2-c,-2+c,2-2c\}$, and if $c-2=2$, we get
traces $\{2,-6,10\}$, but $-6\notin \{0,2,-2\}$ for $p\geq 5$).
\par
In the second case, all elements of $H$ except $\pm 1$ have trace $0$.
We split in two subcases, but depending on properties of $\Fp_p$. 
\par
The first one is when $-1$ is \emph{not} a square in $\Fp_p$.
Conjugating again, we can assume that $g_0\in H$. Because $H$
generates $\SL_2(\Fp_p)$, there exists $g\in H$ which is not $\pm 1$,
$\pm g_0$. If 
$$
g=\begin{pmatrix}
a&b\\
c&-a
\end{pmatrix}\in H
$$
is such an element, we have $a\not=0$, since 
otherwise $b=-c^{-1}$ and the trace of $g_0g$ is $c+c^{-1}$, which is
not in $\{-2,0,2\}$ (non-zero because $-1$ is not a square in our
first subcase), so $\nfold{H}{2}_{nreg}\not=\emptyset$, contrary to
the assumption. Moreover, we can find $g$ as above with $b\not=c$:
otherwise, it would follow that $H$ is contained in the normalizer of
a non-split maximal torus, again contradicting the assumption that $H$
is a generating set.
\par
Now we argue with $g$ as above (i.e., $a\not=0$, $b\not=c$). We have
$$
g_0g=\begin{pmatrix}
c&-a\\
-a& -b
\end{pmatrix}\in \nfold{H}{2},
$$
with non-zero trace $t=c-b$. Moreover, if $t=2$ , i.e., $c=b+2$, the
condition $\det(g_0g)=1$ implies
$$
-2b-b^2-a^2=1
$$
or $(b+1)^2=-a^2$. Similarly, if $t=-2$, we get $(b-1)^2=-a^2$. Since
$a\not=0$, it follows in both cases that $-1$ is a square in $\Fp_p$,
which contradicts our assumption in the first subcase.
\par
Now we come to the second subcase when $-1=z^2$ is a square in
$\Fp_p$. We can then diagonalize $g_0$ over $\Fp_p$, and conjugating
again, this means we can assume that $H$ contains
$$
g'_0=\begin{pmatrix}z&0\\0&-z
\end{pmatrix}
$$
as well as some other matrix
$$
g'=\begin{pmatrix}
a&b\\
c&-a
\end{pmatrix}
$$
(the values of $a$, $b$, $c$ are not the same as before; we are still
in the case when every element of $H$ has trace $0$ except for $\pm
1$). 
\par
Now the trace of $g'_0g'\in \nfold{H}{2}$ is $2za$. But we can find
$g'$ with $a\not=0$, since otherwise $H$ would again not be a
generating set, being contained in the normalizer of the diagonal
(split) maximal torus, and so this trace is non-zero.
\par
The condition $2za=\pm 2$ would give $za=\pm 1$, which leads to
$-a^2=1$. But since $1=\det(g')=-a^2-bc$, we then get $bc=0$ for all
elements of $H$. Finally, if all elements of $H$ satisfy $b=0$, the
set $H$ would be contained in the subgroup of upper-triangular
matrices. So we can find a matrix in $H$ with $b\not=0$, hence
$c=0$. Similarly, we can find another
$$
g''=\begin{pmatrix}
a&0\\
c&-a
\end{pmatrix}
$$
in $H$ with $c\not=0$. Taking into account that $z^2=-1$, computing
the traces of $g'g''$ and of $g'_0g'g''$ gives
$$
bc-2,\quad\quad bcz
$$
respectively. If $bc=2$, the third trace (of an element in
$\nfold{H}{3}$) is $2z\notin \{0,2,-2\}$ since $p\not=2$, and if
$bc=4$, it is $4z\notin \{0,2,-2\}$ since $p\not=5$. And of course if
$bc\notin \{2,4\}$, the first trace is already not in $\{-2,0,2\}$. So
we are done...
\end{proof}

For the proof of Theorem~\ref{th-lp}, we will use the method suggested
by Larsen and Pink at the beginning of~\cite[\S 4]{larsen-pink}.  We
consider the map
$$
\phi
\lmap{\bcl(g)\times \bcl(g)\times \bcl(g)}{\Gg\times \Gg}
{(x_1,x_2,x_3)}{(x_1x_2,x_1x_3)}
$$
and we note that for $(x_1,x_2,x_3)\in (\bcl(g)\cap H)^3$, we have
$\phi(x_1,x_2,x_3)\in \nfold{H}{2}\times\nfold{H}{2}$. We then hope
that the fibers $\phi^{-1}(y_1,y_2)$ of $\phi$ are all finite with
size bounded independently of $(y_1,y_2)\in \Gg\times\Gg$, say of size
at most $c_1\geq 1$. The reason behind this hope is that $\bcl(g)^3$
and $\Gg^2$ have the same dimension, and hence unless something
special happens, we would expect the fibers to have dimension $0$,
which corresponds to having fibers of bounded size since everything is
defined using polynomial equations.
\par
If this hope turns out to be justified, we can count $|\bcl(g)\cap H|$
by summing according to the values of $\phi$: denoting $Z=(\bcl(g)\cap
H)^3$ and $W=\phi(Z)=\phi((\bcl(g)\cap H)^3)$, we have
$$
|\bcl(g)\cap H|^3=|Z|=\sum_{(y_1,y_2)\in W}{ |\phi^{-1}(y_1,y_2)\cap Z|
}
$$
which -- under our optimistic assumption -- leads to the estimate
$$
|\bcl(g)\cap H|^3\leq c_1 |W|\leq c_1|\nfold{H}{2}|^2\leq
c_1\alpha^2|H|,
$$
which has the form we want.
\par
To implement this -- and solve the complications that arise --, we are
led to analyze the fibers of the map $\phi$. The resulting
computations were explained to the author by R. Pink, and start with
an easy observation:

\begin{lemma}\label{lm-biject-lp}
Let $k$ be any field, and let $G=\SL_2(k)$. Let $C\subset G$ be a
conjugacy class, and define
$$
\phi
\lmap{C^3}{G^2}{(x_1,x_2,x_3)}{(x_1x_2,x_1x_3)}.
$$
\par
Then for any $(y_1,y_2)\in G\times G$, we have a bijection
$$
\lmap{C\cap y_1C^{-1}\cap y_2C^{-1}} {\phi^{-1}(y_1,y_2)}
{x_1}{(x_1,x_1^{-1}y_1,x_1^{-1}y_2)}.
$$
\par
In particular, if $k=\bar{\Fp}_p$ and $C$ is a regular semisimple
conjugacy class, we have a bijection
$$
\phi^{-1}(y_1,y_2)\lra C\cap y_1C\cap y_2C.
$$
\end{lemma} 

\begin{proof}
  Taking $x_1$ as a parameter, any $(x_1,x_2,x_3)$ with
  $\phi(x_1,x_2,x_3)=(y_1,y_2)$ can certainly be written
  $(x_1,x_1^{-1}y_1,x_1^{-1}y_2)$. Conversely, such an element in
  $\SL_2(k)^3$ really belongs to $C^3$ (hence to the fiber) if and
  only if $x_1\in C$, $x_1^{-1}y_1\in C$, $x_1^{-1}y_2\in C$, i.e., if
  and only if $x_1\in C\cap y_1C^{-1}\cap y_2C^{-1}$, which proves the
  first part.
\par
For the second part, we need only notice that if $C$ is a regular
semisimple conjugacy class, say that of $g$, then $C=C^{-1}$ because
$g^{-1}$ has the same characteristic polynomial as $g$, hence is
conjugate to $g$.
\end{proof}

We are now led to determine when an intersection of the form $C\cap
y_1C\cap y_2C$ can be infinite. The answer is as follows, and it is
one place where the use of the infinite group $\SL_2(\bar{\Fp}_p)$ is
significant:

\begin{lemma}[Pink]\label{lm-pink}
  Let $k$ be an algebraically closed field of characteristic not equal
  to $2$, and let $g\in \SL_2(k)$ be a regular semisimple element, $C$
  the conjugacy class of $g$. For $y_1$, $y_2\in G$, the intersection
  $X=C\cap y_1C\cap y_2C$ is finite, containing at most two elements,
  unless one of the following cases holds:
\par
\emph{(1)} We have $y_1=1$, or $y_2=1$ or $y_1=y_2$.
\par
\emph{(2)} There exists a conjugate $\Bb=x\Bb_0 x^{-1}$ of the
subgroup
$$
\Bb_0=\Bigl\{
\begin{pmatrix}
a& b\\0& a^{-1}
\end{pmatrix}
\Bigr\}\subset \SL_2(k)
$$
and an element $t\in \Bb\cap C$ such that
\begin{equation}\label{eq-y2}
  y_1,y_2\in \Uu\cup t^2\Uu
\end{equation}
where 
$$
\Uu=x\Uu_0 x^{-1},\quad \quad \Uu_0=\Bigl\{
\begin{pmatrix}
1& b\\0& 1
\end{pmatrix}
\Bigr\}\subset \Bb_0.
$$
\par
In that case, we have $X\subset C\cap \Bb$.
\par
\emph{(3)} The trace of $g$ is $0$.
\end{lemma}

The proof will be given at the end of this section: it is mostly
computational. Before coming back to the proof of Theorem~\ref{th-lp},
we state and prove another preliminary lemma, which is another case of
non-concentration inequalities.

\begin{lemma}\label{lm-subkey}
  For a prime $p$ and $\gamma\in \bar{\Fp}_p^{\times}$, define
$$
C_{\gamma}=\Bigl\{
\begin{pmatrix}
  \gamma & t\\
  0&\gamma^{-1}
\end{pmatrix}
\,\mid\, t\in\bar{\Fp}_p
\Bigr\}.
$$
\par
For any $p\geq 3$, any $\gamma\in\bar{\Fp}_p^{\times}$, any $x\in
\SL_2(\bar{\Fp}_p)$ and any symmetric generating set $H$ of
$\SL_2(\Fp_p)$ containing $1$, we have
$$
|H\cap xC_{\gamma}x^{-1}|=
\Bigl|H\cap
x\Bigl\{
\begin{pmatrix}
\gamma & t\\
0&\gamma^{-1}
\end{pmatrix}\,\mid\, t\in\bar{\Fp}_p\Bigr\}x^{-1} \Bigr|\leq
2\alpha^2 |H|^{1/3}
$$
where $\alpha=\trp{H}$.
\end{lemma}

\begin{proof}
  We first deal with the fact that $x$ and $\gamma$ are not
  necessarily in $\SL_2(\Fp_p)$. We have $xC_{\gamma}x^{-1}\cap
  \SL_2(\Fp_p)\subset x\Bb_0x^{-1}\cap \SL_2(\Fp_p)$, and there are
  three possibilities for the latter: either $x\Bb_0x^{-1}\cap
  \SL_2(\Fp_p)=1$, or $x\Bb_0x^{-1}\cap \SL_2(\Fp_p)=T$ is a non-split
  maximal torus of $\SL_2(\Fp_p)$, or $x\Bb_0x^{-1}\cap
  \SL_2(\Fp_p)=B$ is an $\SL_2(\Fp_p)$-conjugate of the group
  $B_0=\Bb_0\cap \SL_2(\Fp_p)$ of upper-triangular matrices (this is
  once more a standard property of linear algebraic groups over finite
  fields; the most direct argument in this special case is probably to
  observe that we only need to know that $x\Bb_0x^{-1}\cap
  \SL_2(\Fp_p)$ is a subset of a maximal torus, or of a conjugate of
  $B$, which follows from the fact that this intersection is a
  solvable subgroup of $\SL_2(\Fp_p)$). 
\par
In the last case, we can assume that $x\in\SL_2(\Fp_p)$ and $\gamma\in
\Fp_p$. In the first, of course, there is nothing to do. And as for
the second, note that $\gamma$ and $\gamma^{-1}$ are the eigenvalues
of any element in $\SL_2(\Fp_p)\cap xC_{\gamma}x^{-1}$, and there are
at most two elements in a maximal torus with given eigenvalues. A
fortiori, we have $|H\cap xC_{\gamma}x^{-1}|\leq 2\leq 2\alpha^2
|H|^{1/3}$ in that case.
\par
Thus we are left with the situation where $x\in \SL_2(\Fp_p)$.  Using
$\SL_2(\Fp_p)$-conjugation, it is enough to deal with the case
$x=1$. Then either the intersection is empty (and the result is true)
or we can fix
$$
g_0=\begin{pmatrix}
\gamma & t_0\\
0&\gamma^{-1}
\end{pmatrix}\in H\cap C_{\gamma},
$$
and observe that for any $g\in H\cap C_{\gamma}$, we have
$$
g_0^{-1}g\in \nfold{H}{2}\cap C_1,
$$
hence
$$
|H\cap C_{\gamma}|\leq |\nfold{H}{2}\cap C_1|=|\nfold{H}{2}\cap
\Uu_0|,
$$
which reduces further to the case $\gamma=1$.
\par
In that case we have another case of the Larsen-Pink non-concentration
inequality, in that case in a one-dimensional variety.  There is here
also a rather short proof: we fix any element $h\in H$ such that $h$
is \emph{not} in $\Bb_0$, i.e.
$$
h=\begin{pmatrix}a&b\\c&d
\end{pmatrix}
$$
with $c\not=0$. This element exists, because otherwise $H\subset
\Bb\cap \SL_2(\Fp_p)$ would not be a generating set of $\SL_2(\Fp_p)$.
\par
Now consider the multiplication map
$$
\psi\,:\, \lmap{\Uu^*\times\Uu^*\times
  \Uu^*}{\Gg}{(u_1,u_2,u_3)}
{u_1hu_2h^{-1}u_3}
$$
where $\Uu^*=\Uu_0-1$ (we explain below why we do not use $\Uu_0^3$ as
domain).
\par
Note that since $h\in H$, we have $\psi((\Uu^*\cap
\nfold{H}{2})^3)\subset \nfold{H}{8}$. Crucially, we claim that for
\emph{any} $x\in \Gg$, the fiber $\psi^{-1}(x)$ is either empty or
reduced to a single element! If this is true, we get as before
$$
|\Uu^*\cap \nfold{H}{2}|^3\leq |\nfold{H}{8}|\leq \alpha^{6}|H|,
$$
and therefore
$$
|\Uu_0\cap \nfold{H}{2}|=|\Uu^*\cap \nfold{H}{2}|+1\leq
2\alpha^2|H|^{1/3},
$$
which is the result.
\par
To check the claim, we compute. Precisely, if
$$
u_i=\begin{pmatrix}1&t_i\\0&1
\end{pmatrix}\in \Uu^*,
$$
a matrix multiplication leads to
$$
\psi(u_1,u_2,u_3)=
\begin{pmatrix}
1-t_1t_2c^2-t_2ac&\star\\
-t_2c^2&\star
\end{pmatrix},
$$
and in order for this to be a fixed matrix $x$, we see that $t_2$
(i.e., $u_2$) is uniquely determined (since $c\not=0$). Since $u_2$ is
in $\Uu^*$, it is not $1$, and this means that $t_2\not=0$ (ensuring
this is the reason that $\psi$ is defined using $\Uu^*$ instead of
$\Uu_0$). Thus $t_1$ (i.e. $u_1$) is also uniquely determined, and
finally
$$
u_3=(u_1hu_2h^{-1})^{-1}x
$$
is uniquely determined.
\end{proof}


\begin{proof}[Proof of Theorem~\ref{th-lp}]
  We have $g$ regular semisimple with $\Tr(g)\not=0$. We define as
  above the map $\phi$ and denote
$$
Z=(\bcl(g)\cap H)^3,\quad 
W=\phi(Z)=\phi((\bcl(g)\cap H)^3),
$$
so that
\begin{equation}\label{eq-conj-h}
|\bcl(g)\cap H|^3=\sum_{(y_1,y_2)\in W}{ |\phi^{-1}(y_1,y_2)\cap Z|}
=S_0+S_1+S_2,
\end{equation}
where $S_i$ denotes the sum restricted to a subset $W_i\subset W$,
$W_0$ being the subset where the fiber has order at most $2$, while
$W_1$, $W_2$ correspond to those $(y_1,y_2)$ where cases (1) and (2)
of Lemma~\ref{lm-pink} hold. Precisely, we do not put into $W_2$ the
$(y_1,y_2)$ for which both cases (1) and (2) are valid, e.g., $y_1=1$,
and we \emph{add} to $W_1$ the cases where $y_1=-1$, which may
otherwise appear in Case (2). We will prove:
\begin{gather*}
  S_0\leq 2|\nfold{H}{2}|^2\leq 2\alpha^2|H|^2,\quad\quad
  S_1\leq 4|\nfold{H}{2}|^2\leq 4\alpha^2 |H|^2,\\
  S_2\leq 32\alpha^{34/3}|H|^{5/3}.
\end{gather*}
\par
Assuming this, we get immediately
$$
|\bcl(g)\cap H|\leq 6^{2/3}\alpha^{2/3}|H|^{2/3}+2^{5/3}\alpha^{34/9}
|H|^{5/9}
$$
from~(\ref{eq-conj-h}). Now either the second term is smaller than the
first, and we get~(\ref{eq-lp1}) (since $2\cdot 6^{2/3}<7$), or 
$$
2^{5/3}\alpha^{34/9}|H|^{5/9}>6^{2/3}\alpha^{2/3}|H|^{2/3}>
2^{5/3}\alpha^{2/3}|H|^{2/3},
$$
which gives 
$$
\alpha>|H|^{1/28},
$$
the second alternative~(\ref{eq-lp2}) of Theorem~\ref{th-lp}, which is
therefore proved.
\par
We now check the bounds on $S_i$.  The case of $S_0$ follows by the
fact that the fibers over $W_0$ have at most two elements, hence also
their intersection with $Z$, and that $|W_0|\leq |W|\leq
|\nfold{H}{2}|^2$.
\par
The case of $S_1$ splits into four almost identical subcases,
corresponding to $y_1=1$, $y_1=-1$ (remember that we added this,
borrowing it from Case (2)), $y_2=1$ or $y_1=y_2$. We deal only
with the first, say $S_{1,1}$: we have
$$
S_{1,1}\leq \sum_{y_2\in \nfold{H}{2}}{
|\phi^{-1}(1,y_2)\cap Z|
}.
$$
\par
But using Lemma~\ref{lm-biject-lp}, we have
$$
|\phi^{-1}(1,y_2)\cap Z|
=|\{(x_1,x_1^{-1},x_1^{-1}y_2)\in (\bcl(g)\cap H)^3\}|
\leq |\nfold{H}{3}|
$$
for any given $y_2\in \nfold{H}{2}$, since $x_1\in H$ determines the
triple $(x_1,x_1^{-1},x_1^{-1}y_2)$. Therefore
$$
S_{1,1}\leq |\nfold{H}{2}||H|\leq |\nfold{H}{2}|^2,
$$
and similarly for the other three cases.
\par
Now for $S_2$. Here also we sum over $y_1$ first, which is $\not=\pm
1$ (by our definition of $W_2$).
The crucial point is then that an element $y_1\not=\pm1$ is included
in at most two conjugates of $\Bb_0$. Hence, up to a factor $2$, the
choice of $y_1$ fixes that of the relevant conjugate $\Bb$ for which
Case (2) applies. Next we observe that $C_{\Bb}=\bcl(g)\cap \Bb$ is a
conjugate of the union
$$
C_{\alpha}\cup C_{\alpha^{-1}},
$$
where, as in Lemma~\ref{lm-subkey}, we define
$$
C_{\alpha}=\Bigl\{
\begin{pmatrix}
\alpha & t\\
0&\alpha^{-1}
\end{pmatrix}\,\mid\, t\in\bar{\Fp}_p \Bigr\},
$$
and $\alpha$ is such that $\alpha+\alpha^{-1}=\Tr(g)$. Given $y_1\in
\nfold{H}{2}$ and $\Bb$ containing $y_1$, we have by~(\ref{eq-y2})
$$
y_2\in (\nfold{H}{2}\cap \Uu)\cup (\nfold{H}{2}\cap t^2\Uu)
$$
for some $t\in C_{\Bb}$. We note that $t^2\Uu$ is itself conjugate to
$C_{\alpha^2}$ or $C_{\alpha^{-2}}$.
\par
Then the size of the fiber $\phi^{-1}(y_1,y_2)\cap Z$ is determined by
the number of possibilities for $x_1$. As the latter satisfies
$$
x_1\in C_{\Bb}\cap H,
$$
we see that we must estimate the size of intersections of the type
$$
H\cap C_{\gamma},\ \nfold{H}{2}\cap C_{\gamma}
$$
for some fixed $\gamma\in\Fp_p^{\times}$, as this will lead us to
estimates for the number of possibilities for $y_2$ as well as
$x_1$. Using twice Lemma~\ref{lm-subkey}, we get
$$
|\{y_2\,\mid\, (y_1,y_2)\in W_2\}|\leq
8\trp{\nfold{H}{2}}^2|\nfold{H}{2}|^{1/3}\leq 
8\alpha^{25/3}|H|^{1/3},
$$
(the factor $8$ accounts for the two possible choices of $\Bb$ and the
two ``components'' for $y_2$, and the factor $2$ in the lemma) and
$$
|\phi^{-1}(y_1,y_2)\cap Z|\leq 4 \alpha^2|H|^{1/3}.
$$
\par
This gives
$$
S_2\leq 32\alpha^{31/3}|H|^{2/3}|\nfold{H}{2}| \leq
32\alpha^{34/3}|H|^{5/3},
$$
as claimed.
\end{proof}

There now only remains to prove Lemma~\ref{lm-pink}.

\begin{proof}[Proof of Lemma~\ref{lm-pink}]
  It will be convenient to compute the intersection $C\cap
  y_1^{-1}C\cap y_2^{-1}C$ instead of $C\cap y_1C\cap y_2C$, a change
  of notation whichs is innocuous.  
\par
The computation is then based on a list of simple checks. We can
assume that the regular semisimple element $g$ is
$$
g=\begin{pmatrix}
\alpha&0\\
0&\alpha^{-1}
\end{pmatrix}
$$
where $\alpha^4\not=1$, because $\alpha=\pm 1$ implies that $g$ is not
regular semisimple, and $\alpha$ a fourth root of unity implies that
$\Tr(g)=0$, which is the third case of the lemma (recall that $k$ is
assumed to be algebraically closed). Thus the conjugacy class $C$ is
the set of matrices of trace equal to $t=\alpha+\alpha^{-1}$.
\par
The only trick involved is that, for any $y_1\in \SL_2(k)$ and $x\in
\SL_2(k)$, we have
$$
C\cap (xy_1x^{-1})^{-1}C=x(x^{-1}C\cap y_1^{-1}x^{-1}C)=x(C\cap
y_1^{-1}C)x^{-1}
$$
since $x^{-1}C=Cx^{-1}$, by definition of conjugacy classes. This
means we can compute $C\cap y_1^{-1}C$, up to conjugation, by looking
at $C\cap (y_1')^{-1}C$ for any $y'_1$ in the conjugacy class of
$y_1$. In particular, of course, determining whether $C\cap y_1^{-1}C$
is infinite or not only depends on the conjugacy class of $y_1$.
\par
The conjugacy classes in $\SL_2(k)$ are well-known. We will run
through representatives of these classes in order, and determine the
corresponding intersection $C\cap y_1^{-1}C$. Then, to compute $C\cap
y_1^{-1}C\cap y_2^{-1}C$, we take an element $x$ in $C\cap y_1^{-1}C$,
compute $y_2x$, and $C\cap y_1^{-1}C\cap y_2^{-1}C$ corresponds to
those $x$ for which the trace of $y_2x$ is also equal to $t$.
\par
We assume $y_1\not=\pm 1$. Then we distinguish four cases:
\begin{gather}
y_1=\begin{pmatrix}1&1\\0&1
\end{pmatrix},\quad
y_1=\begin{pmatrix}-1&1\\0&-1
\end{pmatrix},\nonumber\\
y_1=\begin{pmatrix}\beta&0\\0&\beta^{-1}
\end{pmatrix},\quad \beta\not=\pm 1,\beta\not=\alpha^{\pm 2}
\label{eq-y1-semisimple}\\
y_1=\begin{pmatrix}\alpha^2&0\\0&\alpha^{-2}
\end{pmatrix}.\nonumber
\end{gather}
\par
We claim that $D=C\cap y_1^{-1}C$ is then given, respectively, by the
sets containing all matrices of the following forms, parameterized by
an element $a\in k$ (with $a\not=0$ in the third case):
\begin{gather}
\begin{pmatrix}
\alpha & a\\0&\alpha^{-1}
\end{pmatrix}\text{ or }
\begin{pmatrix}
\alpha^{-1} & a\\0&\alpha
\end{pmatrix},\label{eq-first-form}\\
\begin{pmatrix}
a & (-a^2+at-1)/(2t)\\
2t & t-a
\end{pmatrix},\nonumber
\end{gather}
\begin{gather}
\frac{1}{\beta+1}\begin{pmatrix}
t & (\beta-\alpha^2)a\\
-(\beta-\alpha^{-2})a^{-1} & t\beta
\end{pmatrix},\label{eq-second-form}
\\
\begin{pmatrix}
\alpha^{-1} & a\\0&\alpha
\end{pmatrix}\text{ or }
\begin{pmatrix}
\alpha^{-1} & 0\\a&\alpha
\end{pmatrix}.\label{eq-last-form}
\end{gather}
\par
Let us check, for instance, the third and fourth cases (cases (1) and
(2) are left as exercise), which we can do simultaneously, taking
$y_1$ as in~(\ref{eq-y1-semisimple}) but without assuming $\beta\not=
\alpha^{\pm 2}$. For
$$
x=\begin{pmatrix}a&b\\c&d
\end{pmatrix}\in C,
$$
we compute
$$
y_1x=\begin{pmatrix}
\beta a& \beta b\\
\beta^{-1} c& \beta^{-1}d.
\end{pmatrix}
$$
\par
This matrix belongs to $C$ if and only if $\beta
a+\beta^{-1}d=t=a+d$. This means that $(a,d)$ is a solution of the
linear system
$$
\begin{cases}
a+d = t\\
\beta a+\beta^{-1} d=t,
\end{cases}
$$
of determinant $\beta^{-1}-\beta\not=0$, so that we have
$$
a=\frac{t}{\beta+1},\quad\quad d=\frac{\beta t}{\beta+1}. 
$$
\par
Write $c=c'/(\beta+1)$, $b=b'/(\beta+1)$; then the condition on $c'$
and $b'$ to have $\det(x)=1$ can be expressed as
$$
-c'b'=(\beta-\alpha^2)(\beta-\alpha^{-2}).
$$
\par
This means that either $\beta$ is not one of $\alpha^2$, $\alpha^{-2}$
(the third case), and then $c$ and $d$ are non-zero, and we can
parametrize the solutions as in~(\ref{eq-second-form}), or else (the
fourth case) $c$ or $d$ must be zero, and then we get upper or
lower-triangular matrices, as described in~(\ref{eq-last-form}).
\par
Now we intersect $D$ (in the general case again) with $y_2^{-1} C$. We
write
$$
y_2=\begin{pmatrix}
x_1&x_2\\
x_3&x_4
\end{pmatrix}.
$$
\par
We consider the first of our four possibilities now, so that $x\in D$
is upper-triangular with diagonal coefficients $\alpha$, $\alpha^{-1}$
(as a set), see~(\ref{eq-first-form}). We compute the trace of $y_2x$,
and find that is
$$
ax_3+x_1\alpha+x_4\alpha^{-1},\text{ or }
ax_3+x_1\alpha^{-1}+x_4\alpha.
$$
\par
Thus, if $x_3\not=0$, there is at most one value of $a$ for which the
trace is $t$, i.e., $D\cap y_2^{-1}C$ has at most two  elements (one
for each form of the diagonal). If $x_3=0$, we find that $x_1$ is a
solution of
$$
\alpha x_1+\alpha^{-1}x_1^{-1}=t,
$$
or 
$$
\alpha x_1^{-1}+\alpha^{-1} x_1=t,
$$
for which the solutions are among $1$, $\alpha^2$ and $\alpha^{-2}$,
so that $y_2$ is upper-triangular with diagonal coefficients $(1,1)$,
$(\alpha^2,\alpha^{-2})$ or $(\alpha^{-2},\alpha^2)$, and this is one
of the instances of Case (2) of Lemma~\ref{lm-pink}.
\par
Let us now consider the second of our four cases, leaving this time
the third and fourth to the reader. Thus we take $x$ as
in~(\ref{eq-second-form}), and compute the trace of $y_2x$ as a
function of $a$, which gives
$$
\Tr(y_2x)=-\frac{x_3}{2t}a^2+\Bigl(x_1-x_4+\frac{x_3}{2}\Bigr)a+
(x_4+2x_2)t.
$$
\par
The equation $\Tr(y_2x)=t$ has therefore at most two solutions, unless
$x_3=0$ and $x_4=x_1$. In that case we have $x_4=1$, and the constant
term is equal to $t$ if and only if $x_4=1$ and $x_2=0$ (so $y_2=1$)
or $x_4=\pm 1$ and $x_2=1$ (and then $y_2=y_1$). Each of these
possibilities corresponds to the exceptional situation of Case (1) of
Lemma~\ref{lm-pink}.
\par
All in all, going through the remaining situations, we finish the
proof.
\end{proof}

\subsection{Proof of Helfgott's Theorem}

We now prove Theorem~\ref{th-helfgott}.  If $p\leq 5$, one checks
numerically that trivial bounds already imply the theorem. So we
assume that $p\geq 7$, which means that Lemma~\ref{lm-escape} is
applicable. We will show that
\begin{equation}\label{eq-precise-goal}
\trp{H}\geq 2^{-1/2} |H|^{1/1512}
\end{equation}
for $p\geq 7$, unless $\nfold{H}{3}=\SL_2(\Fp_p)$, where the latter
case will arise by applying Proposition~\ref{pr-qr}. Then using
Lemma~\ref{lm-small-p}, we derive
$$
\trp{H}\geq \max(2^{1/2},2^{-1/2} |H|^{1/1512})\geq |H|^{1/3024},
$$
which is the precise form of Helfgott's Theorem we claimed.
\par 
We define $\tilde{H}=\nfold{H}{2}$, so that (by Lemma~\ref{lm-escape})
there exists at least one maximal torus $\Tt$ involved with
$\nfold{H}{3}$, hence a fortiori involved with
$L=\nfold{\tilde{H}}{2}=\nfold{H}{4}$.
\par
If, among all maximal tori involved with $L$, there is one for which
the lower bound~(\ref{eq-key-ineq}) (applied to $H=L$) fails, we
obtain from Proposition~\ref{pr-dichotomy} the lower bound
$$
\trp{L}\geq |L|^{1/168}\geq |H|^{1/168},
$$
and since $\trp{L}\leq \alpha^9$ by Ruzsa's Lemma, we get
\begin{equation}\label{eq-first-tripling-bound}
  \alpha\geq |H|^{1/1512}\geq 2^{-1/2} |H|^{1/1512},
\end{equation}
which is~(\ref{eq-precise-goal}).
\par
Otherwise, we distinguish two cases.
\par
\medskip
\par
\underline{Case (1).} There exists a maximal torus $\Tt$ involved with
$L$ such that, for any $g\in G$, the torus $g\Tt g^{-1}$ is involved
with $L$.
\par
As we can guess from~(\ref{eq-key-ineq})
and~(\ref{eq-dist-tori-inter}), in that case, the set $L$ will tend to
be rather large, so $|L|$ is close to $|G|$, \emph{unless} the
tripling constant is itself large enough. 
\par
Precisely, writing $T=\Tt\cap G$, we note that the maximal tori
$$
gTg^{-1}=(g\Tt g^{-1})\cap G
$$
are distinct for $g$ taken among representatives of $G/N_G(T)$. Then
we have the inequalities
$$
|\nfold{L}{2}|\geq \sum_{g\in G/N_G(T)}{ |\nfold{L}{2}\cap g\Tt_{reg}
  g^{-1}| }\geq 7^{-1}\beta^{-4}|L|^{1/3}\frac{|G|}{|N_G(T)|}
$$
where $\beta=\trp{L}$, since each $g\Tt g^{-1}$ is involved with $L$
and distinct regular semisimple elements lie in distinct maximal tori
(and we are in a case where~(\ref{eq-key-ineq}) holds for all tori
involved with $L$).
\par
Now we unwind this inequality in terms of $H$ and $\alpha=\trp{H}$. We
have $\nfold{L}{2}=\nfold{H}{8}$, so
$$
|H|\geq \alpha^{-6}|\nfold{L}{2}|\geq
14^{-1}\alpha^{-6}\beta^{-4}(p-1)^2|L|^{1/3} \geq
14^{-1}\alpha^{-6}\beta^{-4}(p-1)^2|H|^{1/3}
$$
by Ruzsa's Lemma. Furthermore, we have 
$$
\beta=\trp{L}=\trp{\nfold{H}{4}}\leq \alpha^{10}
$$
by Ruzsa's Lemma again, and hence the inequality gives the bound
$$
|H|\geq 14^{-3/2}\alpha^{-69}(p-1)^3,
$$
which for $p\geq 5$ implies $|H|\geq 100^{-1}\alpha^{-69}|G|$.  But
then either
\begin{equation}\label{eq-bound-1}
  \trp{H}=\alpha\geq 200^{-1/69}|G|^{1/621}\geq 2^{-1/2}
  |H|^{1/621},
\end{equation}
or else
$$
|H|\geq 2|G|^{8/9},
$$
which (via Proposition~\ref{pr-qr}) are versions of the two
alternatives we are seeking (in particular the first
implies~(\ref{eq-precise-goal}).)
\par
\medskip
\par
\par
\underline{Case (2).} Since we know that \emph{some} torus is involved
with $L$, the complementary situation to Case (1) is that there exists
a maximal torus $\Tt$ involved with $L=\nfold{H}{4}$ and a conjugate
$g\Tt g^{-1}$, for some $g\in G$, which is \emph{not} involved with
$L$. We are then going to get growth using
Lemma~\ref{lm-intersection}. There is a first clever observation (the
idea of which goes back to work of Glibichuk and Konyagin~\cite{gk} on
the ``sum-product phenomenon''): one can assume, possibly after
changing $\Tt$ and $g$, that $g$ is in $H$.
\par
Indeed, to check this claim, we start with $\Tt$ and $h$ as
above. Since $H$ is a generating set, we can write
$$
g=h_1\cdots h_m
$$
for some $m\geq 1$ and some elements $h_i\in H$. Now let $i\leq m$
be the smallest index such that the maximal torus
$$
\Tt'=(h_{i+1}\cdots h_m)\Tt (h_{i+1}\cdots h_m)^{-1}
$$
is involved with $L$. Taking $i=m$ means that $\Tt$ is involved with
$L$, which is the case, and therefore the index $i$ exists. Moreover
$i\not=0$, again by definition. It follows that
$$
(h_ih_{i+1}\cdots h_m)\Tt (h_ih_{i+1}\cdots h_m)^{-1}
$$
is not involved with $L$. But this means that we can replace $(\Tt,g)$
with $(\Tt',h_i)$, and since $h_i\in H$, this gives us the claim.
\par
We now write $h$ for the conjugator $g$ such that $L$ and the torus
$\Ss=g\Tt g^{-1}=h\Tt h^{-1}$ are not involved. Apply
Lemma~\ref{lm-intersection} with $(H,K)=(\tilde{H},h\Tt h^{-1}\cap G)$
and $n=5$. This gives
$$
\frac{|\nfold{\tilde{H}}{6}|}{|\tilde{H}|} \geq
\frac{|\nfold{\tilde{H}}{5}\cap S|}{|\nfold{\tilde{H}}{2}\cap S|}.
$$
\par
But since $L=\nfold{\tilde{H}}{2}$ and $\Ss$ are not involved (by
construction), we have $|\nfold{\tilde{H}}{2}\cap S|\leq 2$, by the
easy part of the Key Proposition~\ref{pr-dichotomy}, and therefore
$$
\frac{|\nfold{\tilde{H}}{6}|}{|\tilde{H}|} \geq
\frac{1}{2}|\nfold{\tilde{H}}{5}\cap \Ss|.
$$
\par
However, $L$ and $\Tt$ \emph{are} involved, and moreover
$$
h(\nfold{H}{8}\cap \Tt)h^{-1}
\subset \nfold{H}{10}\cap \Ss=\nfold{\tilde{H}}{5}\cap \Ss,
$$
so that
$$
|\nfold{\tilde{H}}{5}\cap \Ss|
\geq |\nfold{H}{8}\cap T|=|\nfold{L}{2}\cap T|\geq
14^{-1}\tilde{\alpha}^{-4}|L|^{1/3}
$$
where $\tilde{\alpha}=\trp{L}$, by the Key
Proposition~\ref{pr-dichotomy} (again, because~(\ref{eq-key-ineq})
holds for all tori involved with $L$).
\par
Thus
$$
\frac{|\nfold{\tilde{H}}{6}|}{|\tilde{H}|}
\geq 28^{-1}\tilde{\alpha}^{-4}|H|^{1/3},
$$
which translates to 
$$
\alpha^{10}|H|\geq 28^{-1}\alpha^{-36}|H|^{4/3},
$$
by Ruzsa's Lemma. This is a rather stronger bound for $\alpha$ than
before, namely
\begin{equation}\label{eq-bound-2}
  \alpha=\trp{H}\geq 28^{-1/46}|H|^{1/138}\geq 2^{-1/2} |H|^{1/138}.
\end{equation}
\par
To summarize, we have obtained three possible lower bounds of the
right kind for $\alpha$,
namely~(\ref{eq-first-tripling-bound}),~(\ref{eq-bound-1})
and~(\ref{eq-bound-2}), one of which holds if
$\nfold{H}{3}\not=\SL_2(\Fp_p)$. All imply~(\ref{eq-precise-goal}),
and hence we are done.

\subsection{Diameter bound}

Corollary~\ref{cor-babai} is a well-known consequence of the growth
theorem: by induction on $j\geq 1$, we see using Helfgott's Theorem
that given a symmetric generating set $S\subset G=\SL_2(\Fp_p)$,
either $\diam\cayley{G}{S}\leq 3^j$, or
$$
|\nfold{H}{3^j}|\geq |H|^{(1+\delta)^j}
$$
where $H=S\cup \{1\}$.  For
$$
j=\Bigl\lceil \frac{\log \log |G|}{\log(1+\delta)}\Bigr\rceil,
$$
the second alternative is impossible, and hence
$$
\diam \cayley{G}{S}\leq 3^j\leq 3(\log |G|)^{(\log 3)/\log(1+\delta)},
$$
which gives the result since $(\log 3)/\log(1+1/3024)\leq 3323$.

\section{The Bourgain-Gamburd method}

The method of Bourgain and Gamburd~\cite{bg} leads, from Helfgott's
growth theorem, to a proof that the Cayley graphs modulo primes of a
Zariski-dense subgroup of $\SL_2(\Zz)$ form an expander
family. Applying this method straightforwardly with explicit estimates
(as done in~\cite[Ch. 4]{expanders}), one obtains explicit expansion
bounds (either for the spectral gap of the combinatorial Laplace
operator, or for the discrete Cheeger constant). However, these
constants are typically very small. 

\subsection{The $L^2$-flattening inequality}

This section applies -- in principle -- to all finite groups, and the
basic expansion criterion that we derive (Corollary~\ref{cor-bg},
following essentially Bourgain and Gamburd) is also of independent
interest.
\par
In rough outline -- and probabilistic language --, the idea is to show
that if two independent $\SL_2(\Fp_p)$-valued symmetrically
distributed random variables $X_1$ and $X_2$ are not too concentrated,
but also not very uniformly distributed on $\SL_2(\Fp_p)$, then their
product $X_1X_2$ will be significantly more uniformly distributed,
\emph{unless} there are obvious reasons why this should fail to
hold. These exceptional possibilities can then be handled separately.
\par
Applying this to some suitable step $X_k$ of the random walk (where
the initial condition is obtained by different means), this result
leads to successive great improvements of the uniformity of the
distribution for $X_{2k}$, $X_{4k}$, \ldots, $X_{2^jk}$, until the
assumptions of the lemma fail. In that situation, the index $m=2^jk$
is of size about $\log |G|$, and $\proba(X_{2m}=1)$ gives a suitable
upper-bound on the number of cycles to obtain expansion, by a variant
of what might be called the Huxley-Sarnak-Xue method
(see~\cite{huxley} and~\cite{sarnak-xue}), as we now recall.

\begin{remark}
  In an earlier draft, we had claimed a much better bound (roughly
  exponentially better) by using non-dyadic steps, but this was due to
  a bad mistake which was pointed out by the referee, which we
  heartily thank once more.
\end{remark}

For a finite group $G$, we denote my $\mdim{G}$ the minimal dimension
of a non-trivial irreducible unitary representation of $G$. Moreover,
if $X$ is a $G$-valued symmetrically-distributed random variable, we
define the \emph{return probability} $\rp{X}$ by
$$
\rp{X}=\proba(X_1X_2=1),
$$
where $(X_1,X_2)$ are independent random variables with the same
distribution as $X$, or equivalently
$$
\rp{X}=\sum_{g\in G}{\proba(X=g)^2}.
$$
\par
Let $S$ be a symmetric generating subset of $G$ and
$\Gamma=\cayley{G}{S}$ the associated Cayley graph.  The Markov
operator $M$ acts on functions on $G$ by
$$
M\varphi(g)=\frac{1}{|S|}\sum_{s\in S}{\varphi(gs)},
$$
and it is a self-adjoint operator. The spectral gap of $G$, as we
normalize it, is equal to $1-\rho^+_{\Gamma}$, where $\rho^+_{\Gamma}$
is the largest eigenvalue of $M$, and it is therefore $\geq
1-\rho_{\Gamma}$, where $\rho_{\Gamma}$ is the spectral radius of $M$.
\par
By expressing spectrally the number of closed walks of length $2m$
from the origin in $\Gamma$, and relating the latter with the return
probability $\rp{X_m}$, where $(X_m)$ is the random walk on the graph
governed by $M$, one gets
$$
\sum_{\rho}{\rho^{2m}}=\frac{1}{|G|}\rp{X_m}.
$$
\par
Using positivity and the fact that $G$ acts without invariant vector
on the $\rho_{\Gamma}$-th eigenspace of $M$, it follows that
$$
\mdim{G}\rho_{\Gamma}^{2m}\leq \frac{1}{|G|}\rp{X_m},
$$
or in other words, we have a bound for the spectral radius in terms of
the return probability: for any $m\geq 1$, we have
\begin{equation}\label{eq-spectral}
1-\lambda_1(\Gamma)\leq
\rho_{\Gamma}\leq \Bigl(\frac{|G|}{\mdim{G}}\rp{X_m}\Bigr)^{1/(2m)}.
\end{equation}
\par
We consider now two independent (not necessarily
identically-distributed) $G$-valued random variables $X_1$, $X_2$ and
let
$$
\rpp{X_1,X_2}=\max(\rp{X_1},\rp{X_2}).
$$
\par
We attempt to bound $\rp{X_1X_2}$ in terms of $\rpp{X_1,X_2}$.  To do
this while still remaining at a level of great generality, the
following definition will be useful:

\begin{definition}[Flourishing]\label{def-delta-control}
  For $\delta>0$, a finite group $G$ is $\delta$-flourishing if any
  symmetric subset $H\subset G$, containing $1$, which generates $G$
  and has tripling constant $\trp{H}<|H|^{\delta}$ satisfies
  $\nfold{H}{3}=G$. 
\end{definition}

In particular, Theorem~\ref{th-helfgott} states that all groups
$\SL_2(\Fp_p)$, for $p$ prime, are $1/3024$-flourishing.
\par
We will prove a general $L^2$-flattening theorem, which may be of
general interest. In order to somehow streamline the proof, we do not
explicitly describe here what ``$G$ large enough'' means. However, all
relevant steps where a condition on the size of $G$ occurs are clearly
marked, and in the second part of Section~\ref{ssec-summary}, we will
look back to express these as explicit inequalities.

\begin{theorem}[$L^2$-flattening conditions]\label{th-l2-flatten}
  Let $G$ be a finite group which is $\delta$-flourishing for some
  $\delta$ with $0<\delta\leq 1$. Let $X_1$, $X_2$ be symmetric
  independent $G$-valued random variables.
\par
Let $0<\gamma<1$ be given, and assume that
\begin{equation}\label{eq-bg-cond3}
  \proba(X_1\in xH)\leq |G|^{-\gamma}
\end{equation}
for all proper subgroups $H\subset G$ and all $x\in G$.
\par
Then for any $\eps>0$, there exists $\delta_1>0$ and $c_3>0$,
depending only on $\eps$, $\delta$ and $\gamma$, such that
\begin{equation}\label{eq-large-enough-rpp}
\rp{X_1X_2}\leq c_3\max\{\frac{1}{|G|^{1-\eps}},
\frac{\rpp{X_1,X_2}}{|G|^{\delta_1}}\Bigr\}
\end{equation}
when $|G|$ is large enough in terms of $(\eps,\delta,\gamma)$.
\par
More precisely, one may take
\begin{equation}\label{eq-delta1}
  \delta_1=\frac{1}{2}\min\Bigl(\frac{\delta\gamma}{2c_2+1},
  \frac{\eps}{2c_2}
  \Bigr)
\end{equation}
where $c_2=937$ is as in Theorem~\ref{th-large-energy} and
$$
c_3\leq 2^{14}c_1\leq 2^{2438}.
$$
\end{theorem}

\begin{proof}
  By definition, we have
$$
\rp{X_1X_2}=\sum_{g\in G} \proba(X_1X_2=g)^2.
$$
\par
We now decompose the ranges of the distribution functions
$$
\nu_i(x)=\proba(X_i=x)
$$ 
into dyadic intervals. Consider a parameter $I\geq 1$, to be chosen
later, and decompose
$$
[\min \proba(X=x),\max \proba(X=x)]\subset [0,1]= \mathcal{I}_0\cup
\mathcal{I}_1\cup\cdots \cup\mathcal{I}_I
$$
where 
$$
\mathcal{I}_i=
\begin{cases}
]2^{-i-1},2^{-i}],&\text{ for } 0\leq i<I\\
[0,2^{-I}]&\text{ for } i=I.
\end{cases}
$$
\par
This gives two partitions of $G$ in subsets
$$
A_{j,i}=\{x\in G \,\mid\, \nu_j(x)=\proba(X_j=x)\in \mathcal{I}_i\},
$$
for $j=1$, $2$. We note that
\begin{equation}\label{eq-size-ai}
  |A_{j,i}|\leq 2^{i+1}
\end{equation}
for $j=1$, $2$ and $0\leq i<I$.
\par
Using this decomposition into the formula above, and the fact that
$$
\proba(X_1X_2=g, X_1\in A_{1,I}\text{ or } X_2\in A_{2,I})\leq
\proba(X_1\in A_{1,I})+\proba(X_2\in A_{2,I})\leq \frac{|G|}{2^{I-1}},
$$
we obtain
\begin{align*}
  \rp{X_1X_2}&= \sum_{g\in G} \Bigl(\sum_{0\leq i,j\leq I}
  \proba(X_1X_2=g, X_1\in A_{1,i},\ X_2\in A_{2.j})\Bigr)^2\\
  &\leq 8|G|^32^{-2I}+ 2\sum_{g\in G}\Bigl(\sum_{0\leq i,j<I}
  \proba(X_1X_2=g, X_1\in A_{1,i},\ X_2\in A_{2,j})\Bigr)^2 \\
  &\leq 2^{3-2I}|G|^3+ 2I^2\sum_{0\leq i,j<I} \sum_{g\in G}
  \proba(X_1X_2=g, X_1\in A_{1,i},\ X_2\in A_{2,j})^2
\end{align*}
by the Cauchy-Schwarz inequality.  Furthermore, the inner sum (say,
$B(A_{1,i},A_{2,j})$) in the second term is given by
\begin{align*}
  B(A_{1,i},A_{2,j})&=\sum_{g\in G} \proba(X_1X_2=g, \ X_1\in
  A_{1,i},\ X_2\in A_{2,j})^2
  \\
  &=\sum_{g\in G}{ \Bigl( \sum_{\stacksum{(x,y)\in A_{1,i}\times A_{2,j}}{xy=g}}{
      \proba(X_1=x)\proba(X_2=y)}\Bigr)^2 }\\
  &=\sum_{\stacksum{x_1,x_2\in A_{1,i}, y_1,y_2\in A_{2,j}}{x_1y_1=x_2y_2}}
  \nu_1(x_1)\nu_1(x_2)\nu_2(y_1)\nu_2(y_2)\\
  & \leq 2^{-2i-2j} |\{(x_1,x_2,y_1,y_2)\in A_{1,i}^2\times
  A_{2,j}^2\,\mid\, x_1y_1=x_2y_2\}|\\&=
2^{-2i-2j}E(A_{1,i},A_{2,j})
\end{align*}
where $E(A,B)$ denotes the multiplicative energy. 
\par
Thus we have proved that 
\begin{equation}\label{eq-step1}
\rp{X_1X_2} \leq 2^{3-2I}|G|^3+2I^2\sum_{0\leq
  i,j<I}2^{-2(i+j)}E(A_{1,i},A_{2,j}).
\end{equation}
\par
We now want to get upper-bounds in terms of the return probability
$\rpp{X_1,X_2}$. This is done in different ways, depending on the size
of the subsets $A_{1,i}$, $A_{2,j}$. We recall first the ``trivial''
bounds
\begin{equation}\label{eq-eab-1}
E(A,B)\leq \min(|A|^2|B|,|A||B|^2).
\end{equation}
\par
We claim that for all $i$ and $j$, we have
\begin{equation}\label{eq-bound1}
2^{-2(i+j)}E(A_{1,i},A_{2,j})\leq 2^4\rpp{X_1,X_2}e(A_{1,i},A_{2,j}),
\end{equation}
and that, for all $\alpha\geq 1$, we have
\begin{equation}\label{eq-bound2}
2^{-2(i+j)}E(A_{1,i},A_{2,j})\leq \alpha^{-1}\rpp{X_1,X_2}
\end{equation}
unless 
\begin{equation}\label{eq-size-ais}
\frac{|A_{1,i}|}{2^i}\geq \frac{1}{2\sqrt{\alpha}},\quad\quad 
\frac{|A_{2,j}|}{2^j}\geq \frac{1}{2\sqrt{\alpha}}.
\end{equation}
\par
To see~(\ref{eq-bound1}), we remark that
\begin{align*}
\rpp{X_1,X_2}
&\geq \frac{1}{2}(\rp{X_1}+\rp{X_2})
= 
\frac{1}{2}\sum_{g\in G}{(\proba(X_1=g)^2+\proba(X_2=g)^2)} \\
&\geq
\frac{1}{2}\Bigl(
\frac{|A_{1,i}|}{2^{2+2i}}+\frac{|A_{2,j}|}{2^{2+2j}} \Bigr) \geq
\frac{1}{4}\frac{(|A_{1,i}||A_{2,j}|)^{1/2}}{2^{i+j}}.
\end{align*}
for any choice of $i$ and $j$. Hence we get
\begin{align*}
  2^{-2(i+j)}E(A_{1,i},A_{2,j})
  &=2^{-2(i+j)}e(A_{1,i},A_{2,j})(|A_{1,i}||A_{2,j}|)^{3/2}
  \\
  &\leq 4\rpp{X_1,X_2}e(A_{1,i},A_{2,j})
  \frac{|A_{1,i}||A_{2,j}|}{2^{i+j}} 
\\&\leq
  16\rpp{X_1,X_2}e(A_{1,i},A_{2,j})
\end{align*}
by~(\ref{eq-size-ai}). 
\par
As for~(\ref{eq-bound2}), if we assume that
$2^{-2(i+j)}E(A_{1,i},A_{2,j})>\alpha^{-1}\rpp{X_1,X_2}$, then we
write simply
$$
2^{-2(i+j)}|A_{1,i}|^2|A_{2,j}|\geq
2^{-2(i+j)}E(A_{1,i},A_{2,j})\geq \alpha^{-1}\frac{|A_{2,j}|}{2^{2+2j}},
$$
using~(\ref{eq-eab-1}), and get the first inequality
of~(\ref{eq-size-ais}), the second being obtained symmetrically.
\par
With these results, we now fix some parameter $\alpha\geq 1$, and let
$$
P_{\alpha}=\{(i,j)\,\mid\, 0\leq i,j<I,\quad |A_{1,i}|\geq
2^{i-1}\alpha^{-1} \text{ and } |A_{2,j}|\geq 2^{j-1}\alpha^{-1}\}.
$$
\par
For $(i,j)\notin P_{\alpha}$, we have
$$
2^{-2(i+j)}E(A_{1,i},A_{2,j})\leq \alpha^{-2}\rpp{X_1,X_2}
$$
by~(\ref{eq-bound2}) and~(\ref{eq-size-ais}), and thus
from~(\ref{eq-step1}), we have shown that
$$
\rp{X_1X_2}\leq 2^{3-2I}|G|^3+2\alpha^{-2}\rpp{X_1,X_2}I^4
+32\rpp{X_1,X_2}I^2 \sum_{(i,j)\in P_{\alpha}}e(A_{1,i},A_{2,j})
$$
(estimating the size of the complement of $P_{\alpha}$ by $I^2$). We
select
$$
I=\left\lceil \frac{2\log 2|G|}{\log 2}\right\rceil\leq 3\log (3|G|),
$$
and hence obtain 
$$
  \rp{X_1X_2}\leq \frac{1}{|G|}
  +2^8\rpp{X_1,X_2}(\log 3|G|)^2
  \Bigl\{\frac{(\log 3|G|)^2}{\alpha^2}
  +2\sum_{(i,j)\in P_{\alpha}}e(A_{1,i},A_{2,j})
  \Bigr\}.
$$
\par
We apply this bound with $\alpha=|G|^{\delta_0}$, where $\delta_0>0$
will be chosen later. Thus
\begin{multline*}
  \rp{X_1X_2}\leq \frac{1}{|G|} +2^8\rpp{X_1,X_2}(\log
  3|G|)^4|G|^{-2\delta_0}+ \\
  2^9(\log 3|G|)^2\rpp{X_1,X_2}\sum_{(i,j)\in
    P_{\alpha}}e(A_{1,i},A_{2,j}).
\end{multline*}
\par
Let then
$$
R_{\alpha}=\{(i,j)\in P_{\alpha}\,\mid\, e(A_{1,i},A_{2,j})\geq
\alpha^{-1}\}\subset P_{\alpha},
$$
so that the contribution of those $(i,j)\in P_{\alpha}$ which are not
in $R_{\alpha}$, together with the middle term, can be bounded by
$$
\frac{2^{13}(\log 3|G|)^4}{|G|^{\delta_0}}\rpp{X_1,X_2}.
$$
\par
We can now analyze the set $R_{\alpha}$; it turns out to be very
restricted when $\delta_0$ is chosen small enough.  By
Theorem~\ref{th-large-energy}, for each $(i,j)\in R_{\alpha}$, there
exists a $\beta_1$-approximate subgroup $\app{H}_{i,j}$ and elements
$(x_i,y_j)\in A_{1,i}\times A_{2,j}$ such that
$$
|\app{H}_{i,j}|\leq \beta_2|A_{1,i}|,\quad\quad |A_{1,i}\cap
x_i\app{H}_{i,j}|\geq \beta_3^{-1}|A_{1,i}|,\quad\quad |A_{2,j}\cap
\app{H}_{i,j}y_j|\geq \beta_3^{-1}|A_{2,j}|,
$$
and with tripling constant bounded by $\beta_4$, where the $\beta_i$
are bounded qualitatively by
$$
\beta_i\leq c_1|G|^{c_2\delta_0}
$$
for some absolute constants, which we take to be $c_1=2^{2424}$,
$c_2=937$ using~(\ref{eq-explicit-mc}). We then note first that if
$H_{i,j}$ denotes the ``ordinary'' subgroup generated by
$\app{H}_{i,j}$, we have
\begin{multline}
\proba(X_1\in x_iH_{i,j})\geq \proba(X_1\in x_i\app{H}_{i,j}) 
\geq
\\\proba(X_1\in A_{1,i}\cap x_i\app{H}_{i,j}) 
\geq
\frac{1}{\beta_3}\frac{|A_{1,i}|}{2^{i+1}}\geq \frac{1}{4\beta_3 \alpha}
\geq \frac{1}{4c_1|G|^{(1+c_2)\delta_0}},
\label{eq-large-enough1}
\end{multline}
where we used the definition of $P_{\alpha}$. If $\delta_0$ is small
enough that
\begin{equation}\label{eq-small1}
(1+c_2)\delta_0<\gamma,
\end{equation}
and if $|G|$ is large enough, this is not compatible
with~(\ref{eq-bg-cond3}), and we can therefore assume that each
$\app{H}_{i,j}$ (if any!) generates the group $G$.
\par
We next observe that $\app{H}_{i,j}$ can not be extremely
small. Indeed, we have
$$
|\app{H}_{i,j}|\geq |x_i\app{H}_{i,j}\cap A_{1,i}|\geq \beta_3^{-1}|A_{1,i}|,
$$
on the one hand, and by applying~(\ref{eq-bg-cond3}) with $H=1$, we
can see that $A_{1,i}$ is not too small, namely
$$
|A_{1,i}|\geq \frac{\proba(X_1\in A_{1,i})}{\max_{g\in
    G}\proba(X_1=g)}\geq |G|^{\gamma}\proba(X_1\in A_{1,i}) \geq
\frac{|G|^{\gamma}|A_{1,i}|}{2^{i+1}}\geq \frac{|G|^{\gamma}}{4\alpha}
$$
using again the definition of $P_{\alpha}$. 
\par
This gives the lower bound
\begin{equation}\label{eq-hij-not-too-small}
|\app{H}_{i,j}|\geq \frac{|G|^{\gamma}}{4\alpha\beta_3}\geq
\frac{1}{4c_1}|G|^{\gamma_1}
\end{equation}
with $\gamma_1=\gamma-\delta_0(1+c_2)$ (which is $>0$
by~(\ref{eq-small1})), and then leads to control of the tripling
constant, namely
\begin{equation}\label{eq-large-enough2}
\trp{\app{H}_{i,j}}\leq \beta_4 \leq c_1|G|^{c_2\delta_0} \leq
c_1(4c_1)^{2\delta_0\gamma_1^{-1}}|\app{H}_{i,j}|^{c_2\delta_0\gamma_1^{-1}}.
\end{equation}
\par
Since we assumed that $G$ is $\delta$-flourishing, we see from
Definition~\ref{def-delta-control} that if $\delta_0$ is such that
\begin{equation}\label{eq-small2}
  \frac{c_2\delta_0}{\gamma_1}=\frac{c_2\delta_0}{\gamma-(1+c_2)\delta_0}
  <\delta,
\end{equation}
and again if $|G|$ is large enough, the approximate subgroup
$\app{H}_{i,j}$ must in fact be very large, specifically it must
satisfy
$$
\app{H}_{i,j}\cdot \app{H}_{i,j}\cdot \app{H}_{i,j}=G,
$$
and in particular
$$
|\app{H}_{i,j}|\geq \frac{|G|}{\beta_4}\geq
\frac{1}{c_1}|G|^{1-c_2\delta_0}.
$$
\par
Intuitively, this implies that $X_1$ and $X_2$ are already rather
uniformly distributed over $G$, and hence that $\rpp{X_1,X_2}$ is
already too small to be significantly improved at the level of
$X_1X_2$. To express this idea concretely, we go back to the first
stage of the argument, namely~(\ref{eq-step1}): the contribution to
$\rp{X_1X_2}$ coming from $(i,j)$ was bounded by
$$
2^{-2(i+j)}E(A_{1,i},A_{2,j})\leq \frac{|A_{1,i}||A_{2,j}|^2}{2^{2(i+j)}}
\leq \frac{1}{2^{i-3}}
$$
by~(\ref{eq-size-ai}). But then we also have
$$
2^{i+1}\geq |A_{1,i}|
\geq \frac{|\app{H}_{i,j}|}{\beta_2} \geq
\frac{|G|}{\beta_2\beta_4} \geq c_1^{-1}|G|^{1-c_2\delta_0},
$$
(observe that $\beta_2\beta_4\leq c_1|G|^{c_2\delta_0}$)
and therefore
$$
2^{-2(i+j)}E(A_{1,i},A_{2,j})\leq 16c_1|G|^{-1+2c_2\delta_0}.
$$
\par
Using again the trivial bound $I^2\leq 9(\log 3|G|)^2$ for the number
of possible pairs $(i,j)$ to which this applies, the conclusion is an
inequality
\begin{equation}\label{eq-large-enough3}
  \rp{X_1X_2}\leq \frac{1}{|G|}+
  2^{11}c_1\frac{(\log 3|G|)^4}{|G|^{1-c_2\delta_0}}
  +2^{13}\frac{(\log 3|G|)^4}{|G|^{\delta_0}}\rpp{X_1,X_2},
\end{equation}
which holds (under the assumptions that $|G|$ is sufficiently large)
for all $\delta_0$ small enough so that~(\ref{eq-small1})
and~(\ref{eq-small2}) are satisfied. It is elementary
that~(\ref{eq-small2}) is stronger than~(\ref{eq-small1}) and is
equivalent with
$$
\delta_0<\frac{\delta\gamma}{(1+\delta)c_2+\delta},
$$
which holds when $\delta_0<\delta\gamma/(2c_2+1)$ (since we assume
$\delta\leq 1$).
\par 
Thus we can apply this for
$$
\delta_0=\min\Bigl(\frac{\delta\gamma}{2c_2+1}, \frac{\eps}{2c_2}
\Bigr)=2\delta_1,
$$
where $\delta_1$ is given by~(\ref{eq-delta1}). Then for $|G|$ large
enough,~(\ref{eq-large-enough3}) implies~(\ref{eq-large-enough-rpp}),
and hence we have finished the proof of Theorem~\ref{th-l2-flatten}.
\end{proof}




We can summarize all this as follows (with the same remark as before
concerning our handling of the conditions on the size of $G$):

\begin{corollary}[The Bourgain-Gamburd expansion
  criterion]\label{cor-bg}
  Let $\uple{c}=(c,d,\delta,\gamma)$ be a tuple of positive real
  numbers, and let $\bbigf(\uple{c})$ be the family of all finite
  connected Cayley graphs $\cayley{G}{S}$ for which the following
  conditions hold:
\par
\emph{(1)} We have $\mdim{G}\geq |G|^{d}$;
\par
\emph{(2)} The group $G$ is $\delta$-flourishing;
\par
\emph{(3)} For the random walk $(X_n)$ on $G$ with $X_0=1$, we have
that
$$
\proba(X_{k}\in xH)\leq |G|^{-\gamma}
$$
for some $k\leq c\log |G|$ and all $x\in G$ and proper subgroups
$H\subset G$.
\par
Then, 
for any $\Gamma\in \bbigf(\uple{c})$ with $|\Gamma|$ large enough, the
spectral gap of the normalized Laplace operator of $\Gamma$ satisfies
$$
\lambda_1(\Gamma)\geq 1-\exp\Bigl(-\frac{d}{4cj}\Bigr),
$$
where
$$
j\leq 8\max\Bigl( \frac{2c_2+1}{\delta\gamma},\frac{16c_2}{7d} \Bigr).
$$
\end{corollary}

Note that it is not clear at this point that this corollary is not an
empty statement (or one that applies at most to finitely many graphs
with a bounded valency). But in the next section we will check that it
applies to the situation of Theorem~\ref{th-bg-bis} to prove that
certains families of Cayley graphs are expanders.

\begin{proof}
  Let $\Gamma=\cayley{G}{S}$ be a graph in $\bbigf(\uple{c})$.  We
  will apply Theorem~\ref{th-l2-flatten} with $\eps=d/2$ so that
$$
  \delta_1=\frac{1}{2}\min\Bigl(\frac{\delta\gamma}{2c_2+1},
  \frac{d}{4c_2}
\Bigr)
$$
\par
When $|G|$ is large enough, we can rephrase the conclusion using the
simpler inequality
\begin{equation}\label{eq-large-enough4}
\rp{Y_1Y_2}\leq c_3\max\Bigl(\frac{1}{|G|^{1-d/2}},
\frac{\rpp{Y_1,Y_2}}{|G|^{\delta_1}}\Bigr) \leq
\max\Bigl(\frac{1}{|G|^{1-3d/4}},\frac{\rpp{Y_1,Y_2}}{|G|^{\delta_1/2}}\Bigr),
\end{equation}
for random variables $Y_1$, $Y_2$ which satisfy the assumptions of
this theorem.
\par
Let $k=\lfloor c\log|G|\rfloor$ be given by (3). We apply the theorem
to $Y_1=X_{2^jk}$ and $Y_2=X_{2^{(j+1)}k}Y_1^{-1}$ for $j\geq 0$.
These are indeed independent and symmetric random variables, and
Conditions (2) and (3) imply that we can indeed apply
Theorem~\ref{th-l2-flatten} to these random variables for any $j\geq
2$. Since $Y_1$ and $Y_2$ are identically distributed, we have
$$
\rpp{Y_1,Y_2}=\rp{Y_1}=\rp{X_{2^jk}}.
$$
\par
Thus, applying the theorem, we obtain by induction
$$
\rp{X_{2^jk}}\leq \rp{X_{k}}|G|^{-j\delta_1/2}\leq |G|^{-j\delta_1/2}
$$
when $j$ is such that
$$
|G|^{1-3d/4}>|G|^{j\delta_1/2}, 
$$
and for larger $j$, we get
$$
\rp{X_{2^jk}}\leq |G|^{-1+3d/4}.
$$
\par
In particular, we obtain this last inequality for
$$
j=\Bigl\lceil \frac{2(1-3d/4)}{\delta_1}\Bigr\rceil\leq
\frac{4}{\delta_1}\leq 8\max\Bigl(
\frac{2c_2+1}{\delta\gamma},\frac{4c_2}{d}
\Bigr),
$$
which, by the ``cycle-counting'' inequality~(\ref{eq-spectral}), gives
$$
\rho_{\Gamma}\leq (|G|^{1-d}\rp{X_{jk}})^{1/(2^jk)} \leq
\exp\Bigl(-\frac{d}{2^{j+3}c}\Bigr),
$$
which 
thus proves the theorem.
\end{proof}

\subsection{Expansion bounds for $\SL_2$}
\label{sec-implement}

Theorem~\ref{th-bg-bis} will now be proved by applying the criterion
of Corollary~\ref{cor-bg}. Thus we will consider the groups
$G_p=\SL_2(\Fp_p)$ for $p$ prime, for which Condition (1) of the
Bourgain-Gamburd criterion (which is purely a group-theoretic
property) is given by
$$
\mdim{\SL_2(\Fp_p)}=\frac{p-1}{2}
$$
for $p\geq 3$ (a result of Frobenius), which gives a value of $d$
arbitrarily close to $1/3$, for $p$ large enough. Condition (2) is
given by Helfgott's Theorem, with $\delta=1/3024$. Note that it is
purely a property of the groups $\SL_2(\Fp_p)$.
\par
Condition (3), on the other hand, depends on the choice of generating
sets. The symmetric generating sets $S_p$ in Theorem~\ref{th-bg-bis}
are assumed to be obtained by reduction modulo $p$ of a fixed
symmetric subset $S\subset \SL_2(\Zz)$. We will argue first under the
additional assumption that $S\subset \SL_2(\Zz)$ generates a free
group.
\par
We begin with a classical proposition, whose idea goes back to
Margulis. For the statement, recall that the norm of a matrix $g\in
\GL_n(\Cc)$ is defined by
$$
\|g\|=\max_{v,w\not=0}\frac{|\langle gv,w\rangle|}{\|v\|\|w\|}
$$
where $\langle\cdot,\cdot\rangle$ is the standard inner product on
$\Cc^n$. This satisfies
\begin{equation}\label{eq-prop-norm}
\|g_1g_2\|\leq \|g_1\|\|g_2\|,\quad\quad
\max_{i,j}{|g_{i,j}|}\leq \|g\|\text{ for } g=(g_{i,j}),
\end{equation}
the latter because $g_{i,j}=\langle ge_i,e_j\rangle$ in terms of the
canonical basis.

\begin{proposition}[Large girth for finite Cayley graphs]\label{pr-girth}
  Let $S\subset \SL_2(\Zz)$ be a symmetric set, and let
  $\Gamma=\cayley{G}{S}$ be the corresponding Cayley graphs. Let
  $\tau>0$ be defined by
\begin{equation}\label{eq-def-tau}
\tau^{-1}=\log \max_{s\in S}{\|s\|}>0,
\end{equation}
which depends only on $S$.
\par
\emph{(1)} For all primes $p$ and all $r<\tau\log(p/2)$, where
$G_p=\SL_2(\Fp_p)$, the subgraph $\Gamma_r$ induced by the ball of
radius $r$ in $\Gamma$ maps injectively to $\cayley{G_p}{S}$.
\par
\emph{(2)} If $G$ is freely generated by $S$, in particular $1\notin
S$, the Cayley graph $\cayley{G_p}{S}$ contains no cycle of length
$<2\tau\log(p/2)$, i.e., its girth $\girth{\cayley{G_p}{S}}$ is at
least $2\tau\log(p/2)$.
\end{proposition}

\begin{proof}
  The main point is that if all coordinates of two matrices $g_1$,
  $g_2\in \SL_2(\Zz)$ are less than $p/2$ in absolute value, a
  congruence $g_1\equiv g_2\mods{p}$ is equivalent to the equality
  $g_1=g_2$. And because $G$ is freely generated by $S$, knowing a
  matrix in $G$ is equivalent to knowing its expression as a word in
  the generators in $S$. 
\par
Thus, let $x$ be an element in the ball of radius $r$ centered at the
origin. By definition, $x$ can be expressed as
$$
x=s_1\cdots s_m
$$
with $m\leq r$ and $s_i\in S$. Using~(\ref{eq-prop-norm}), we get
$$
\max_{i,j}|x_{i,j}|\leq \|x\|\leq \|s_1\|\cdots \|s_m\|\leq
e^{m/\tau}\leq e^{r/\tau}.
$$
\par
Applying the beginning remark and this fact to two elements $x$ and
$y$ in the ball $\vois{1}{r}$ of radius $r$ centered at $1$, for $r$
such that $e^{r/\tau}<\frac{p}{2}$, it follows that $x\equiv
y\mods{p}$ implies $x=y$, which is (1).
\par
Then (2) follows because any embedding of a cycle $\gamma\,:\, C_m\ra
\cayley{G_p}{S}$ such that $\gamma(0)=1$ and such that
$$
d(1,\gamma(i))\leq m/2<\tau\log (p/2)
$$
for all $i$ can be lifted to the cycle (of the same length) with image
in the Cayley graph of $G$ with respect to $S$, and if $S$ generates
freely $G$, the latter graph is a tree. Thus a cycle of length
$m=\girth{\cayley{G_p}{S}}$ must satisfy $m/2\geq \tau\log(p/2)$.
\end{proof}

We can now check Condition (3) in the Bourgain-Gamburd criterion,
first for cosets of the trivial subgroup, i.e., for the probability
that $X_n$ be a fixed element when $n$ is of size $c\log p$ for some
fixed (but small) $c>0$. As we did earlier, we clearly mark where
we impose conditions on the size of $p$, and these will be made
explicit in Section~\ref{ssec-summary}.

\begin{corollary}[Decay of probabilities]\label{cor-decay}
  Let $S\subset \SL_2(\Zz)$ be a symmetric set, $G$ the subgroup
  generated by $S$. Assume that $S$ freely generates $G$. Let $p$ be a
  prime such that the reduction $S_p$ of $S$ modulo $p$ generates
  $G_p=\SL_2(\Fp_p)$, and let $(X_n)$ be the random walk on
  $\cayley{G_p}{S_p}$ with $X_0=1$. Let
$$
\tau^{-1}=\log \max_{s\in S}{\|s\|}>0,
$$
as in Proposition~\ref{pr-girth}.
\par
Fix a constant $c$ with $0<c\leq 1$. If $p$ is large enough, depending
on $c$ and $S$, then for
$$
n=c\lfloor \tau \log(p/2)\rfloor
$$
and any $x\in \SL_2(\Fp_p)$, we have
\begin{equation}\label{eq-prob-gamma}
\proba(X_n=x)\leq |G_p|^{-c\gamma_1}
\end{equation}
where
\begin{equation}\label{eq-ggamma1}
  \gamma_1=\frac{\tau(\log (\kesten\sqrt{|S|}))}{8}.
\end{equation}
\par
More precisely, this holds for all
\begin{equation}\label{eq-large-enough-ok1}
p\geq \max\Bigl(17, 2\exp\Bigl(\frac{2}{c\tau}\Bigr)\Bigr).
\end{equation}
\end{corollary}

The ``extra'' parameter $c$ will be useful in the argument involving
all proper subgroups $H$ below.

\begin{proof}
  There exists $\tilde{x}\in G$ such that $\tilde{x}$ reduces to $x$
  modulo $p$ and $\tilde{x}$ is at the same distance to $1$ as $x$,
  and by Proposition~\ref{pr-girth}, (2), we have
$$
\proba(X_n=x)=\proba(\tilde{X}_n=\tilde{x}),
$$
for $n\leq \tau\log(p/2)$, where $(\tilde{X}_n)$ is the random walk
starting at $1$ on the $|S|$-regular tree $\cayley{G}{S}$. By a
well-known result of Kesten~\cite{kesten}, we have
$$
\proba(\tilde{X}_n=\tilde{x})\leq r^{-n}\quad\text{with}\quad
r=\frac{|S|}{2\sqrt{|S|-1}},
$$
for all $n\geq 1$ and all $\tilde{x}\in G$. Since $c\leq 1$ we have
$$
n=c\lfloor \tau \log(p/2)\rfloor\geq c\tau\log(p/2)-1,
$$
and we obtain
$$
\proba(X_n=x)\leq r\Bigl(\frac{p}{2}\Bigr)^{-c\tau\log r}\leq
\Bigl(\frac{p}{2}\Bigr)^{-\demi c\tau\log r},
$$
for $p\geq 2r^{2/(c\tau\log r)}$. Using the inequality
$$
\frac{p}{2}\geq |G_p|^{1/4}
$$
for $p\geq 17$, this becomes
$$
\proba(X_n=x)\leq |G_p|^{-c\tau(\log r)/8}
$$
for all $p\geq \max(17, 2r^{2/(c\tau\log r)})$. Since $r\geq
\frac{2}{\sqrt{3}} \sqrt{|S|}$, we get the desired result.
\end{proof}

In order to deal with cosets of other proper subgroups of
$\SL_2(\Fp_p)$, we will exploit the fact that those subgroups are very
well understood, and in particular, there is no proper subgroup that
is ``both big and complicated''. Precisely, by results going back to
Dickson (see, e.g., the account in~\cite[Ch. 6]{suzuki} for
$\PSL_2(\Fp_p)$, from which the result for $\SL_2(\Fp_p)$ follows
easily), one knows that for $p\geq 5$, if $H\subset \SL_2(\Fp_p)$ is a
proper subgroup, one of the following two properties holds:
\par
(1) The order of $H$ is at most $120$;
\par
(2) For all $(x_1,x_2,x_3,x_4)\in H$, we have
\begin{equation}\label{eq-relation-h}
[[x_1,x_2],[x_3,x_4]]=1.
\end{equation}
\par
The first ones are ``small'', and will be easy to handle
using~(\ref{eq-prob-gamma}). The second are, from the group-theoretic
point of view, not very complicated (their commutator subgroups are
abelian). The following \emph{ad-hoc} lemma\footnote{\ Note that this
  is the only place where using prime fields $\Fp_p$ instead of
  arbitrary finite fields really simplifies the argument,
  since~(\ref{eq-relation-h}) does not hold for proper subgroups of,
  say, $\SL_2(\Fp_{p^2})$.} takes care of them:

\begin{proposition}\label{pr-ad-hoc}
  Let $k\geq 2$ be an integer and let $W\subset F_k$ be a subset of
  the free group on $k$ generators $(a_1,\ldots, a_k)$ such that
\begin{equation}\label{eq-two-step}
[[x_1,x_2],[x_3,x_4]]=1
\end{equation}
for all $(x_1,x_2,x_3,x_4)\in W$. Then for any $m\geq 1$, we have
$$
|\{x\in W\,\mid\, d_{T}(1,x)\leq m\}|\leq (4m+1)(8m+1)\leq 45m^2,
$$
where $T$ is the $|S|$-regular tree $\cayley{F_k}{S}$, $S=\{a_i^{\pm
  1}\}$.
\end{proposition}

\begin{proof}
  The basic fact we need is that the condition $[x,y]=1$ is very
  restrictive in $F_k$: precisely, for a fixed $x\not=1$, we have
  $[x,y]=1$ if and only if $y\in C_{F_k}(x)$, which is an infinite
  cyclic group.  Denoting a generator by $z$, we find
\begin{equation}\label{eq-commute-in-a-ball}
  |\{y\in\vois{1}{m}\,\mid\, [x,y]=1\}|=
  |\{h\in \Zz\,\mid\, d_{T_k}(1,z^h)\leq m\}|\leq 2m+1
\end{equation}
since (a standard fact in free groups) we have $d_{T}(1,z^h)\geq |h|$.
\par
Let $W$ be a set satisfying the assumption~(\ref{eq-two-step}), which
we assume to be not reduced to $\{1\}$. We denote $W_m=W\cap
\vois{1}{m}$. First, if $[x,y]=1$ for all $x$, $y\in W_m$, then by
taking a fixed $x\not=1$ in $W_m$, we get $W_m\subset C_{F_k}(x)\cap
\vois{1}{m}$, and~(\ref{eq-commute-in-a-ball}) gives the result.
\par
Otherwise, fix $x_0$ and $y_0$ in $W_m$ such that
$a=[x_0,y_0]\not=1$. Then, for all $y$ in $W_m$ we have
$[a,[x_0,y]]=1$. Noting that $d_T(1,[x_0,y])\leq 4m$, it follows again
from the above that the number of possible values of $[x_0,y]$ is at
most $8m+1$ for $y\in W_m$.
\par
Now for one such value $b=[x_0,y]$, we consider how many $y_1\in W_m$
may satisfy $[x_0,y_1]=b$. We have $[x_0,y]=[x_0,y_1]$ if and only if
$\varphi(y^{-1}y_1)=y^{-1}y_1$, where $\varphi(y)=x_0yx_0^{-1}$
denotes the inner automorphism of conjugation by $x_0$. Hence $y_1$
satisfies $[x_0,y_1]=b$ if and only if $\varphi(y^{-1}y_1)=y^{-1}y_1$,
which is equivalent to $y^{-1}y_1\in C_{F_k}(x_0)$. Since $y^{-1}y_1$
is an element at distance $\leq 2m$ of $1$ if $y$ and $y_1$ are in
$\vois{1}{m}$, applying~(\ref{eq-commute-in-a-ball}) gives
$$
  |\{y_1\in\vois{1}{m}\,\mid\, [x_0,y_1]=[x_0,y]\}|
  \leq 4m+1,
$$
and hence we have $|W_m|\leq (4m+1)(8m+1)$ in that case, which proves
the result.
\end{proof}

Using Corollary~\ref{cor-decay}, we finally verify fully Condition (3)
in Corollary~\ref{cor-bg}:

\begin{corollary}[Decay of probabilities, II]\label{cor-decay-2}
  Let $S\subset \SL_2(\Zz)$ be a symmetric set, $G$ the subgroup
  generated by $S$. Assume that $S$ freely generates $G$. Let $p$ be a
  prime such that the reduction $S_p$ of $S$ modulo $p$ generates
  $G_p=\SL_2(\Fp_p)$, and let $(X_n)$ be the random walk on
  $\cayley{G_p}{S_p}$ with $X_0=1$. Let
$$
\tau^{-1}=\log \max_{s\in S}{\|s\|}>0,
$$
as in Proposition~\ref{pr-girth}.  
\par
If $p$ is large enough, then for 
$$
n=\Bigl\lfloor \frac{\tau}{32}\log(p/2) \Bigr\rfloor,
$$
any $x\in \SL_2(\Fp_p)$ and any proper subgroup $H\subset
\SL_2(\Fp_p)$, we have
\begin{equation}\label{eq-prob-gamma-true}
  \proba(X_n\in xH)\leq |G_p|^{-\gamma}
\end{equation}
where
\begin{equation}\label{eq-def-gamma}
\gamma=\frac{\tau(\log(\kesten\sqrt{ |S|}))}{2^{9}}.
\end{equation}
\end{corollary}

\begin{proof}
  We start by noting that
$$
\proba(X_n\in xH)^2\leq \proba(X_{2n}\in H)
$$
for all $x\in G_p$ and all subgroups $H\subset G_p$.
\par
Consider first the case where~(\ref{eq-relation-h}) holds for $H$. Let
$\tilde{H}\subset G$ be the pre-image of $H$ under reduction modulo
$p$. If $2n\leq \tau\log(p/2)$, then as in the proof of
Corollary~\ref{cor-decay}, we get
$$
\proba(X_{2n}\in H)=\proba(\tilde{X}_{2n}\in \tilde{H}).
$$
\par
Provided $n$ also satisfies the stronger condition $n\leq
m=\tfrac{1}{16}\tau\log(p/2)$, any commutator
$$
[[x_1,x_2],[x_3,x_4]]
$$
with $x_i\in \tilde{H}\cap\vois{1}{n}$ is an element at distance at
most $\tau\log(p/2)$ from $1$ in the tree $\cayley{G}{S}$, which
reduces to the identity modulo $p$ by~(\ref{eq-relation-h}), and
therefore must be itself equal to $1$. In other words, we can apply
Proposition~\ref{pr-ad-hoc} to $W=\tilde{H}\cap \vois{1}{m}$ to deduce
the upper bound
$$
|\tilde{H}\cap \vois{1}{m}|\leq 45 m^2.
$$
\par
We now take
$$
n=\frac{1}{32}\lfloor \tau\log(p/2)\rfloor,
$$
and we derive
$$
\proba(X_{2n}\in H) \leq |\tilde{H}\cap \vois{1}{m}| (\max_{x\in
  G_p}\proba(X_{2n}=x)) \leq 45m^2|G_p|^{-\gamma_1/16}
$$
(where $\gamma_1$ is given by~(\ref{eq-ggamma1}), as in
Corollary~\ref{cor-decay}), and hence
\begin{equation}\label{eq-large-enough5}
\proba(X_n\in xH)\leq \frac{\sqrt{45}}{16}\tau(\log
p/2)|G_p|^{-\gamma_1/32} \leq |G_p|^{-\gamma_1/64}
\end{equation}
provided $p$ is large enough, which is the conclusion in that case.
\par
On the other hand, if~(\ref{eq-relation-h}) does not hold, we have
$|H|\leq 120$, and for the same value of $n$ we get
\begin{equation}\label{eq-large-enough6}
  \proba(X_n\in xH)\leq 120|G_p|^{-\gamma_1/32}\leq |G_p|^{-\gamma_1/64}
\end{equation}
for $p$ large enough, by Corollary~\ref{cor-decay} with $c=1/32$.
This gives again the desired result.
\end{proof}

The following upper-bound on $\gamma$ was suggested by the referee:

\begin{lemma}\label{lm-bound-gamma}
  With notation as in Corollary~\ref{cor-decay-2}, we have
$$
\gamma\leq 2^{-5}.
$$
\end{lemma}

\begin{proof}
For $n\geq 1$, the cardinality of the ball $\vois{1}{n}$ is at least
$(|S|-1)^n$, and is at most
$$
|\{g\in M_2(\Zz)\,\mid\, |g_{i,j}|\leq (\max_{s\in S}{\|s\|})^n,\text{
  for } 1\leq i,j\leq 2\}
$$
by~(\ref{eq-prop-norm}). Thus, denoting $\Delta=\max_{s\in S}{\|s\|}$,
we find
$$
\log (|S|-1)\leq 4\log(2\Delta+1),
$$
and hence
$$
\gamma=2^{-9}\frac{\log(\kesten\sqrt{|S|})}{\log(\Delta)} \leq
2^{-7}\frac{\log(2\Delta+1)}{\log (\Delta)}.
$$
\par
Now we note that either $\Delta\geq 2$, or $S$ is contained in the
finite set of matrices in $\SL_2(\Zz)$ where all coefficients are in
$\{-1,0,1\}$. There are $20$ such matrices, and all those which are
not of finite order are parabolic.  For these, we have $\|s\|\geq
\sqrt{2}$, and therefore $\Delta\geq \sqrt{2}$ in all cases, and hence
$$
\gamma\leq 2^{-7}\frac{\log (2\sqrt{2}+1)}{\log(\sqrt{2})}
\leq 2^{-5}.
$$
\end{proof}

\subsection{Summary}
\label{ssec-summary}

We can now summarize how to obtain an explicit spectral gap, for large
enough $p$, in the situation of Theorem~\ref{th-bg-bis}, finishing the
proof. We then explain how to quantify the condition on $p$.
\par
We first consider the case where $S\subset \SL_2(\Zz)$ freely
generates a free group of rank $\geq 2$ (in which case it is
automatically Zariski-dense in $\SL_2$).
\par
\medskip
\par
\underline{Step 1 (when $p$ is large enough).} We have
$$
\mdim{G_p}=\frac{p-1}{2}
$$
for $p\geq 3$. In particular, $\mdim{G_p}\geq |G_p|^d$ for any $d<1/3$
provided $p$ is large enough in terms of $d$. Moreover, by
Theorem~\ref{th-helfgott}, those groups are $\delta$-flourishing with
$\delta=1/3024$.
\par
For the random walk $(X_n)$ on $G_p$ associated to the generating set
$S_p$, with $X_0=1$, we have
$$
\proba(X_{k}\in xH)\leq |G|^{-\gamma}
$$
when 
$$
k=\Bigl\lfloor \frac{\tau}{32}\log(p/2)\Bigr\rfloor\leq
\frac{\tau}{96}\log(|G_p|)
$$
with 
$$
\tau^{-1}=\log \max_{s\in S}{\|s\|},\quad\quad
\gamma=\frac{\tau\log (\kesten\sqrt{|S|})}{2^{9}}.
$$
by~(\ref{eq-def-tau}) and~(\ref{eq-def-gamma}). Thus in
Corollary~\ref{cor-bg}, we can take $c=1/96$. The number of times we
apply the basic $L^2$-flattening inequality is bounded by
$$
j\leq 8\max\Bigl( \frac{2c_2+1}{\delta\gamma},\frac{4c_2}{d}
\Bigr)\leq 8\max\Bigl( \frac{1875\cdot 3024}{\gamma},15000\Bigr)
=48060000\gamma^{-1}
$$
(using Lemma~\ref{lm-bound-gamma}) and the spectral gap satisfies
$$
\lambda_1(\Gamma)\geq 1-\exp\Bigl(-\frac{d}{2^{j+3}c}\Bigr)
= 1-\exp\Bigl(-\frac{d}{2^{j+3}c}\Bigr)
\geq \frac{d}{2^{j+4}c},
$$
for all $p$ large enough. For $p\geq 17$, we take $d=1/4$, and this
gives
$$
\lambda_1(\Gamma)\geq \frac{d}{2^{j+4}c} \geq \frac{3}{2^{j+1}}\geq
2^{-2^{26}\gamma^{-1}}.
$$
\par
Except that we incorporated the factor $2^9$ from the current value of
$\gamma$ to the constant factor (for esthetic reasons), this
gives~(\ref{eq-stated-sg}).
\par
\medskip
\par
\underline{Step 2 (how large is ``large enough'').}  We gather here,
as a series of inequalities to be satisfied by $p$, the conditions
under which we can apply the previous lower bound. These we gather
from the proofs of the results of this section. First come
inequalities that make explicit the condition that $|G|$ be large
enough in Theorem~\ref{th-l2-flatten}, which are easily translated
into conditions on $p$ since $|\SL_2(\Fp_p)|=p(p^2-1)$.
\begin{itemize}
\item In order that~(\ref{eq-large-enough1})
  contradict~(\ref{eq-bg-cond3}), we must have
$$
|G|^{\gamma-\delta_0(1+c_2)}>4c_1.
$$
\item In order that~(\ref{eq-large-enough2}) contradict the growth
  alternative of Helfgott's Theorem, it is enough that
$$
|G|^{\gamma_1}>4c_1\Bigl\{
c_1(4c_1)^{\gamma_1^{-1}}
\Bigr\}^{(\delta-c_2\delta_0\gamma_1^{-1})^{-1}}
$$
where\footnote{\ This is not the same $\gamma_1$ that occurs in the
  proof of the decay of probabilities.}
$\gamma_1=\gamma-(1+c_2)\delta_0$ (in view
of~(\ref{eq-hij-not-too-small})).
\item In order that~(\ref{eq-large-enough3})
  give~(\ref{eq-large-enough-rpp}) when $\delta_1$
  satisfies~(\ref{eq-delta1}), it is enough that
$$
|G|^{\eps-2c_2\delta_0}\geq (\log 3|G|)^4,
$$
and that
$$
|G|^{\delta_0}\geq c_1^{-2}(\log 3|G|)^4.
$$
\item In order that~(\ref{eq-large-enough4}) hold, we must have
\begin{equation}\label{eq-worse}
\min(|G|^{d/4},|G|^{\delta_1/2})\geq c_3.
\end{equation}
\end{itemize}
\par
Now we list the conditions needed to apply the Bourgain-Gamburd
criterion in the situation of Theorem~\ref{th-bg-bis}, when $S$ freely
generates a free group of rank $|S|/2\geq 2$.
\begin{itemize}
\item We need
$$
p\geq \max\Bigl(17,  2\exp\Bigl(\frac{2}{c\tau}\Bigr)\Bigr)
$$
by~(\ref{eq-large-enough-ok1}).
\item In order that the last inequality in~(\ref{eq-large-enough5})
  hold, as well as~(\ref{eq-large-enough6}), it is enough that
$$
|\SL_2(\Fp_p)|^{\gamma}\geq \max\Bigl(120,\Bigl(\log
\frac{p}{2}\Bigr)\Bigr).
$$
\end{itemize}
\par
\begin{remark}
  Below in Section~\ref{ssec-pari} is found a straightforward
  \textsc{Pari/GP}~\cite{pari} that computes the lower-bound of Step 1
  for the spectral gap, given the set of matrices $S$, and that can
  also be used to determine for which $p$ the bound is known to be
  applicable.
\end{remark}

We finally explain how to reduce the full statement of
Theorem~\ref{th-bg-bis} to the case where the given symmetric subset
$S\subset \SL_2(\Zz)$ generates a free group, which is the one treated
by the Bourgain-Gamburd method.
\par
For a given $S\subset \SL_2(\Zz)$ which generates a Zariski-dense
subgroup $G$ of $\SL_2$, the intersection $G\cap \Gamma(2)$, where
$\Gamma(2)$ is the principal congruence subgroup modulo $2$, is a free
subgroup of finite index in $G$. From a free generating set, one can
extract two generators $s_1$, $s_2\in G$ to obtain a free subgroup of
rank $2$ of $G$, say $G_1$ (since $G\cap \Gamma(2)$ has finite index
in $G$, it is still Zariski-dense, and hence has rank at least
$2$). This subgroup is still Zariski-dense. We can then compare the
expansion for the Cayley graphs of $\SL_2(\Fp_p)$ with respect to $S$
and to $S_1=\{s_1^{\pm 1},s_2^{\pm 1}\}$.
\par
For $p$ large enough so that $G_p=\SL_2(\Fp_p)$ is generated both by
$S$ modulo $p$ and $S_1$ modulo $p$, we have
$$
d(x,y)\leq Cd_1(x,y)
$$
where $d_1(\cdot,\cdot)$ is the distance in the Cayley graph
$\Gamma_1=\cayley{G_p}{S_1\mods{p}}$, and $d(\cdot,\cdot)$ the
distance in $\Gamma_2=\cayley{G_p}{S\mods{p}}$ and $C$ is the maximum
of the word length of $s_1$, $s_2$ with respect to $S$. Hence, by a
standard lemma (see, e.g.,~\cite[Lemma~3.1.16]{expanders},
applied to $\Gamma_1$ and $\Gamma_2$ with $f$ the identity), the
expansion constants satisfy
$$
h(\cayley{G_p}{S\mods{p}})=h(\Gamma_2)
\geq w^{-1}h(\cayley{G_p}{S_1\mods{p}})
$$
with 
$$
w=4\sum_{j=1}^{\lfloor C\rfloor}|S|^{j-1}.
$$
\par
In particular, using Theorem~\ref{th-bg-bis} for $G_1$, we obtain the
expansion property for $G$, and we can bound the spectral gap
explicitly once we know expressions for the generators $s_1$, $s_2$ in
terms of those in $S$.
\par
As the referee pointed out, Breuillard and
Gelander~\cite[Th. 1.2]{breuillard-gelander} have proved a strong
uniform version of the Tits alternative which implies that there
exists an absolute constant $N\geq 1$ such that, for any Zariski-dense
subgroup $G\subset \SL_2(\Zz)$, and for any symmetric generating set
$S\subset G$, the combinatorial ball of radius $N$ in $\cayley{G}{S}$
contains two elements which generate a free subgroup of rank $2$ of
$G$. If a concrete value of $N$ was known (which does not seem to be
the case yet), one could use the above argument to state a version of
the second part of Theorem~\ref{th-bg-bis} without the assumption of
freeness.


\subsection{Diameter bound}

We can now also prove quickly Corollary~\ref{cor-diameter}. Let
$S_1=S\cup\{1\}$. By Proposition~\ref{pr-girth}, if we let
$$
r=\Bigl\lfloor \tau\log\frac{p}{2}\Bigr\rfloor,
$$
where $\tau$ is defined by~(\ref{eq-def-tau}), the size of
$\nfold{S_1}{r}$ is at least the size of a ball of radius $r$ in a
$|S|$-regular tree, which is well-known to be at least $s^r$, where
$s=|S|-1$.
\par
For $p\geq 17$, this gives
$$
\nfold{S_1}{r}\geq s^{-1}\Bigl(\frac{p}{2}\Bigr)^r\geq
s^{-1}|\SL_2(\Fp_p)|^{\tau(\log r)/4},
$$
and if $p\geq \exp(2\tau^{-1})$, this becomes
$$
\nfold{S_1}{r}\geq |\SL_2(\Fp_p)|^{\delta_2},
$$
where
$$
\delta_2=\frac{\tau(\log s)}{8}>0.
$$
\par
Now we apply repeatedly Helfgott's Theorem with
$H=\nfold{S_1}{r}$. For $j$ such that
$$
j\geq \frac{\log (\delta_2^{-1})}{\log(1+\delta)},
$$
the $3^j$-fold product of $H$ must be equal to $\SL_2(\Fp_p)$, and
hence we get
$$
\diam\cayley{\SL_2(\Fp_p)}{S}\leq 3^jr\leq 3^{j-1}(\log
|\SL_2(\Fp_p)|),
$$
and taking
$$
j=\Bigl\lceil \frac{\log (\delta_2^{-1})}{\log(1+\delta)}\Bigr\rceil,
$$
this gives the bound
$$
\diam\cayley{\SL_2(\Fp_p)}{S}\leq
3^{\log(\delta_2^{-1})/\log(1+\delta)}
(\log |\SL_2(\Fp_p)|).
$$



\subsection{Script}\label{ssec-pari}

Here is a \textsc{Pari/GP}~\cite{pari} script that performs the
computations needed to obtain an explicit spectral for
Theorem~\ref{th-bg-bis}, given as input a set of matrices $S$ which
generate a free group (this condition is not checked).
\par
\lstset{breaklines=true,basicstyle=\ttfamily\footnotesize}
\begin{lstlisting}
\\ Norm of a matrix
matnorm(m)=sqrt(sum(i=1,matsize(m)[1],sum(j=1,matsize(m)[2],m[i,j]^2)))

\\ Spectral radius of random walk on k-regular tree
gapr(s)=local(k);k=length(s);k/2/sqrt(k-1)

\\ Growth constant in Helfgott's Theorem
gapdelta(s)=1/3024

\\ Minimal dimension of irreducible, OK for p at least 17
gapd(s)=1/4

\\ Constant c_2 in explicit multiplicative combinatorics
gapc2(s)=937

\\ Logarithm of c_1, base 2
gaplogc1(s)=2424

\\ Logarithm of c_3, base 2
gaplogc3(s)=2438

\\ "tau" invariant
gaptau(s)=1/log(vecmax(vector(length(s),i,matnorm(s[i]))))

\\ Value of gamma for p large enough

gapgamma(s)=gaptau(s)*log(2/sqrt(3)*sqrt(length(s)))/2^9

\\ Bound for minus the logarithm in base 2 of spectral gap
\\ for p large enough

gaploggap(s)=2^26/gapgamma(s)

\\ Value of delta_0
gapdelta0(s)=min(gapdelta(s)*gapgamma(s)/(2*gapc2(s)+1),gapd(s)/8/gapc2(s))

\\ Value of delta_1
gapdelta1(s)=1/2*min(gapdelta(s)*gapgamma(s)/(2*gapc2(s)+1),gapd(s)/8/gapc2(s))

\\ Value of gamma1 in lower-bound conditions
gapgamma1(s)=gapgamma(s)-gapdelta1(s)*(1+gapc2(s))

\\ First minimal value on log p, base 2 
gaplogmin1(s)=(2+gaplogc1(s))/3/(gapgamma(s)-gapdelta1(s)*(1+gapc2(s)))

\\ Second minimal value on log p, base 2
gaplogmin2(s)=1/3/gapgamma1(s)*(2+gaplogc1(s)+1/(gapdelta(s)-gapc2(s)*gapdelta0(s)/gapgamma1(s))*(gaplogc1(s)+1/gapgamma1(s)*(2+gaplogc1(s))))

\\ Is log(p)=lp larger than third minimal value on log p (base e)?
gapislogmin3(s,lp)=if(lp>=1/3/(gapd(s)/4-2*gapc2(s)*gapdelta0(s))*4*(log(log(3)+lp)),1,0)

\\ Is log(p) larger than fourth minimal value on log p (base e)?
gapislogmin4(s,lp)=if(lp>=1/3/gapdelta0(s)*(4*log(log(3)+lp)-2*log(2)*gaplogc1(s)),1,0)

\\ Fifth minimal values on log p, base 2
gaplogmin5(s)=gaplogc3(s)/min(gapd(s)/4,gapdelta1(s)/2)

\\ Constant c used in sixth minimal value
gapc(s)=gaptau(s)/96

\\ Sixth minimal value on log p, base 2
gaplogmin6(s)=max(log(17)/log(2), 1+(2/gaptau(s)/gapc(s))/log(2))

\\ Is log(p) larger than seventh minimal value on log p, base e
gapislogmin7(s,lp)=if(3*lp*gapgamma(s)>=log(lp-log(2)),1,0)

\\ Eighth minimal value on log p, base 2
gaplogmin8(s)=log(120)/log(2)/3/gapgamma(s)

\\ Minimum of log(p), base 2, for gapislogmin3
gapfind3(s)= {
  local(j=2,i,k);
  while(!gapislogmin3(s,j),j=2*j);
  k=j/2;
  i=ceil((j+k)/2);
  while(i!=j,
    if(!gapislogmin3(s,i),
       k=i;i=ceil((j+k)/2),
       j=i;i=ceil((j+k)/2)));
  ceil(i/log(2));
}

\\ Minimum of log(p), base 2, for gapislogmin4
gapfind4(s)= {
  local(j=2,i,k);
  while(!gapislogmin4(s,j),j=2*j);
  k=j/2;
  i=ceil((j+k)/2);
  while(i!=j,
    if(!gapislogmin4(s,i),
       k=i;i=ceil((j+k)/2),
       j=i;i=ceil((j+k)/2)));
  ceil(i/log(2));
}

\\ Minimum of log(p), base 2, for gapislogmin7
gapfind7(s)= {
  local(j=2,i,k);
  while(!gapislogmin7(s,j),j=2*j);
  k=j/2;
  i=ceil((j+k)/2);
  while(i!=j,
    if(!gapislogmin7(s,i),
       k=i;i=ceil((j+k)/2),
       j=i;i=ceil((j+k)/2)));
  ceil(i/log(2));
}

\\ Minimum value of log(p), base 2
gapmin(s)=ceil(vecmax([gaplogmin1(s),gaplogmin2(s),gapfind3(s),gapfind4(s),gaplogmin5(s),gaplogmin6(s),gapfind7(s),gaplogmin8(s)]))

\\ Base 2 bound for gapmin(s)
gapminlog(s)=ceil(log(gapmin(s))/log(2))

\\ Generators of the Lubotzky group
ls=[[1,3;0,1],[1,-3;0,1],[1,0;3,1],[1,0;-3,1]]

\\ ? gaploggap(ls)
\\ gaploggap(ls)
\\ %29 = 49218765900.678024122794300454114797957
\\ ? log(gaploggap(ls))/log(2)
\\ log(gaploggap(ls))/log(2)
\\ %30 = 35.518489433339899156620630706081878314
\\ ? gapminlog(ls)
\\ gapminlog(ls)
\\ %31 = 46
\end{lstlisting}

\section{Appendix: proof of Theorem~\ref{th-large-energy}}\label{app-combin}


In this appendix, we sketch the proof of
Theorem~\ref{th-large-energy}, following very (essentially) line by
line Tao's paper~\cite{tao}. The presentation is therefore highly
condensed, though we use a ``diagram'' notation which should make it
relatively easy to check how the values of the constants evolve.
\par
Below all sets are subsets of a fixed finite group $G$, and are all
non-empty.

\subsection{Diagrams}

We will use the following diagrammatic notation:
\par
\begin{enumerate}
\item If $A$ and $B$ are sets with Ruzsa distance 
$$
d(A,B)=\log\Bigl(\frac{|A\cdot B^{-1}|}{\sqrt{|A||B|}}\Bigr)
$$ 
such that $d(A,B)\leq \log \alpha$, we write
$$
\xymatrix{
A  \rd{\alpha} & B
},
$$
\item If $A$ and $B$ are sets with $|B|\leq \alpha |A|$, we write
$$
\xymatrix{
B \fb{\alpha} &A
},
$$
and in particular if $|X|\leq \alpha$, we write
$\xymatrix{X \fb{\alpha} &1}$,
\item If $A$ and $B$ are sets with $e(A,B)\geq 1/\alpha$, we write
$$
\xymatrix{
A \me{\alpha} & B
},
$$
\item If $A\subset B$, we also write $\xymatrix{A \sub & B}$.
\end{enumerate}
\par
\medskip
\par
The following rules are easy to check (in addition to some more
obvious ones which we do not spell out):
\begin{enumerate}
\item From
$$
\xymatrix{
A\rd{\alpha} & B
}
$$
we can get
$$
\xymatrix{
A\fb{\alpha^2} & B
},\quad
\xymatrix{
B\fb{\alpha^2} & A
}.
$$
\item (Ruzsa's triangle inequality) From
$$
\xymatrix{
A  \rd{\alpha_1} & B \rd{\alpha_2} & C
}
$$
we get
$$
\xymatrix{
A\rd{\alpha_1\alpha_2} & C
}.
$$
\item  From
$$
\xymatrix{
C  \fb{\alpha_1} & B \fb{\alpha_2} & A
}
$$
we get
$$
\xymatrix{
C\fb{\alpha_1\alpha_2} & A
}.
$$
\item (``Unfolding edges'') From
$$
\xymatrix{
B  \fb{\alpha} \crd{\beta} & A
}
$$
we get
$$
\xymatrix{ AB^{-1}\lfb{\sqrt{\alpha}\beta} &&A }
$$
(note that by the first point in this list, we only need to have
$$
\xymatrix{B \rd{\beta} & A}
$$
to obtain the full statement with $\alpha=\beta^2$, which is usually
qualitatively equivalent.)
\item (``Folding'') From
$$
\xymatrix{
AB^{-1}  \fb{\alpha} & A \fb{\beta} & B
}
$$
we get
$$
\xymatrix{ A\lrd{\alpha\beta^{1/2}} &&B }.
$$
\end{enumerate}

Note that the relation $\xymatrix{A\fb{\alpha} &B}$ is purely a matter
of the size of $A$ and $B$, while the other arrow types depend on
structural relations involving the sets (for $\xymatrix{A\sub & B}$)
and product sets (for $\xymatrix{A\rd{\alpha} &B}$ or
$\xymatrix{A\me{\alpha}&B}$).

\subsection{Proofs}

First we state the Ruzsa covering lemma~\cite[Lemma 3.6]{tao} in our
language:

\begin{theorem}[Ruzsa]
If 
$$
\xymatrix{ AB\fb{\alpha} & A },
$$
there exists a set $X$ which satisfies
$$
\xymatrix{
X\sub &B
},\quad
\xymatrix{
X\fb{\alpha}  &1
},\quad
\xymatrix{
B\sub &A^{-1}AX
},
$$
and symmetrically, if
$$
\xymatrix{BA\fb{\alpha} & A},
$$
there exists $Y$ with
$$
\xymatrix{
Y\sub &B
},\quad
\xymatrix{
Y\fb{\alpha}  &1
},\quad
\xymatrix{
B\sub &YAA^{-1}
}.
$$
\end{theorem}




Next we have the link between sets with small tripling and approximate
subgroups~\cite[Th. 3.9 and Cor. 3.10]{tao}:

\begin{theorem}\label{th-app-tripling}
Let $A=A^{-1}$ with $1\in A$ and
$$
\xymatrix{\nfold{A}{3}\fb{\alpha} & A}.
$$
\par
Then $H=\nfold{A}{3}$ is a $(2\alpha^{5})$-approximate subgroup
containing $A$.
\end{theorem}

\begin{proof}
We have first
$$
\xymatrix{H\fb{\alpha} & A},\quad
\xymatrix{A\sub & H}.
$$
\par
Then by Ruzsa's lemma~\ref{pr-ruzsa}, we get
$$
\xymatrix{A\nfold{H}{2}=\nfold{A}{7}\fb{\alpha^{5}}& A},
$$
and by the Ruzsa covering lemma there exists $X$ with
$$
\xymatrix{X\sub & \nfold{H}{2}},\quad \xymatrix{X \fb{\alpha^{5}} & 1},
$$
such that
$$
\xymatrix{\nfold{H}{2}\sub & \nfold{A}{2} X \sub & \nfold{A}{3} X=HX}.
$$
\par 
Taking $X_1=X\cup X^{-1}$, we get a symmetric set with
$$
\xymatrix{X_1\sub & \nfold{H}{2}},\quad \xymatrix{X_1 \fb{2\alpha^{5}} & 1},
$$
and
$$
\xymatrix{\nfold{H}{2}\sub & HX},\quad \xymatrix{\nfold{H}{2}\sub & XH},
$$
which are the properties defining a $(2\alpha^{5})$-approximate
subgroup.
\end{proof}

The next result is the explicit form of~\cite[ Th. 4.6, (i) implies
(ii)]{tao}:

\begin{theorem}\label{th-46}
Let $A$ and $B$ with
$$
\xymatrix{A\rd{\alpha} &B^{-1}}
$$
\par
Then there exists a $\gamma$-approximate subgroup $H$ and a set $X$
with
$$
\xymatrix{X\fb{\gamma_1}& 1},\quad \xymatrix{A\sub &XH},\quad
\xymatrix{B\sub & HX},\quad 
\xymatrix{H\fb{\gamma_2} & A},
$$
where
$$
\gamma\leq 2^{21}\alpha^{80},
\quad\quad
\gamma_1\leq 2^{28}\alpha^{104},
\quad\quad
\gamma_2\leq 8\alpha^{14}.
$$
\par
Furthermore, one can ensure that
\begin{equation}\label{eq-trip-1}
\xymatrix{\nfold{H}{3}\lfb{2^{10}\alpha^{40}}&& H}.
\end{equation}
\end{theorem}

\begin{proof}
From
$$
\xymatrix{
A  \fb{1} \crd{\alpha^2} & A
},
$$
we get first
$$
\xymatrix{AA^{-1}\fb{\alpha^2} & A}.
$$
\par
By~\cite[Prop. 4.5]{tao}, we find a set $S$ with\footnote{\ The
  property $1\in S$ is not explicitly stated in~\cite{tao}, but
  follows from the explicit definition used by Tao, namely $S=\{x\in
  G\,\mid\, |A\cap Ax|>(2\alpha^2)^{-1}|A|\}$.} $1\in S$ and $S=S^{-1}$
such that
$$
\xymatrix{A\fb{2\alpha^2} & S},\quad
\xymatrix{A\nfold{S}{n}A^{-1}\lfb{2^n\alpha^{4n+2}} && A}
$$
for all $n\geq 1$.  In particular, we get
$$
\xymatrix{AS^{-1}=AS\lfb{2\alpha^6}&& A},\quad
\xymatrix{S\lfb{2\alpha^6}&& A}.
$$
\par
We have
$$
\xymatrix{\nfold{S}{3}\lfb{8\alpha^{14}}&& A\lfb{2\alpha^2} && S},
$$
and Theorem~\ref{th-app-tripling} says that $H=\nfold{S}{3}$ is a
$\gamma$-approximate subgroup containing $S$, with
$\gamma=2(16\alpha^{16})^{5}=2^{21}\alpha^{80}$, and (as we see)
$$
\xymatrix{H\lfb{8\alpha^{14}} && A}.
$$
\par
Moreover, we have
$$
\xymatrix{\nfold{H}{3}=\nfold{S}{9}\sub & A\nfold{S}{9}A^{-1} 
\lfb{2^9\alpha^{38}} &&   A\fb{2\alpha^2} & S},
$$
which gives~(\ref{eq-trip-1}).
\par
Now from
$$
\xymatrix{AH=A\nfold{S}{3}\lfb{8\alpha^{14}}&& A\lfb{2\alpha^2} && S
  \fb{1} & H},
$$
we see by the Ruzsa covering lemma that there exists $Y$ with
$$
\xymatrix{Y\sub & A},\quad
\xymatrix{Y\lfb{16\alpha^{16}} && 1},\quad
\xymatrix{A\sub &YHH}.
$$
\par
By definition of an approximate subgroup, there exists $Z$ with
$$
\xymatrix{Z\fb{\gamma} & 1},\quad
\xymatrix{HH\sub & ZH},
$$
and hence
$$
\xymatrix{A \sub & (YZ)H}.
$$
\par
Now we go towards $B$. First we have
$$
\xymatrix{AH^{-1}=A\nfold{S}{3}\lfb{8\alpha^{14}}&& A\fb{2\alpha^2} & H}
$$
which, again by folding, gives
$$
\xymatrix{A\rd{\alpha_1} & H}
$$
with $\alpha_1=8\sqrt{2}\alpha^{15}$. Hence we can write
$$
\xymatrix{H \rd{\alpha_1} & A \rd{\alpha} & B^{-1}},
$$
and so
$$
\xymatrix{H\lrd{\alpha\alpha_1} && B^{-1}}.
$$
\par
In addition, we have
$$
\xymatrix{H\lfb{8\alpha^{14}}&& A\fb{\alpha^2}& B^{-1}},
$$
and therefore we get
$$
\xymatrix{
H  \lfb{8\alpha^{16}} \lcrd{\alpha\alpha_1} && B^{-1},
}
$$
from which it follows by unfolding that
$$
\xymatrix{
B^{-1}H^{-1}=B^{-1}H\lfb{32\alpha^{20}} && B^{-1}
\fb{\alpha^2} & A\fb{2\alpha^2} & H
}.
$$
\par
Once more by the Ruzsa covering lemma, we find $Y_1$ with
$$
\xymatrix{Y_1\sub &B^{-1}},\quad \xymatrix{Y_1\lfb{2^6\alpha^{24}} 
&&  1},
\quad
\xymatrix{B^{-1}\sub & Y_1HH\sub & (Y_1Z)H}.
$$
\par
Now we need only take $X=(Y_1Z\cup YZ)$, so that
$$
\xymatrix{X\fb{\gamma_1}& 1}
$$
with $\gamma_1=\gamma(64\alpha^{24}+16\alpha^{16})$, in order to
conclude. Since
$$
\gamma_1\leq 2^{28}\alpha^{104},
$$
we are done.
\end{proof}

The next result is a version of the Balog-Gowers-Szemer\'edi Lemma
found in~\cite[Th. 5.2]{tao}.

\begin{theorem}
Let $A$ and $B$ with
$$
\xymatrix{A\me{\alpha} &B}.
$$
\par
Then there exist $A_1$, $B_1$ with
$$
\xymatrix{A_1\sub & A},\quad \xymatrix{B_1\sub & B},
$$
as well as
$$
\xymatrix{A\lfb{8\sqrt{2}\alpha} && A_1},\quad
\xymatrix{B\lfb{8\alpha} && B_1},
$$
and
$$
\xymatrix{A_1\rd{\alpha_1} & B_1^{-1}}
$$
where $\alpha_1=2^{23}\alpha^9$.
\end{theorem}

This is not entirely spelled out in~\cite{tao}, but only the last two
or three inequalities in the proof need to be made explicit to obtain
this value of $\alpha_1$. Finally, the next theorem is just the
``diagrammatic'' version of Theorem~\ref{th-large-energy}, and
therefore completes its proof. It is an explicit version
of~\cite[Th. 5.4; (i) implies (iv)]{tao}.

\begin{theorem}\label{th-me-approximate}
Let $A$ and $B$ with
$$
\xymatrix{A\me{\alpha} &B}.
$$
\par
Then there exist a $\beta$-approximate subgroup $H$ and $x$, $y\in G$,
such that
$$
\xymatrix{H\fb{\beta_2} & A},\quad
\xymatrix{A\fb{\beta_1} & A\cap xH},\quad
\xymatrix{B\fb{\beta_1} & B\cap Hy},
$$
where
$$
\beta\leq 2^{1861}\alpha^{720},\quad\quad
\beta_1\leq 2^{2424}\alpha^{937},\quad\quad
\beta_2\leq 2^{325}\alpha^{126}.
$$
\par
Moreover, one can ensure that
$$
\xymatrix{\nfold{H}{3}\fb{\beta_3} & H}
$$
where $\beta_3=2^{930}\alpha^{360}$.
\end{theorem}

\begin{proof}
  By the Balog-Gowers-Szemer\'edi Theorem, we get $A_1$, $B_1$ with
$$
\xymatrix{A_1\sub & A},\quad \xymatrix{B_1\sub & B},
$$
as well as
$$
\xymatrix{A\lfb{8\sqrt{2}\alpha} && A_1},\quad
\xymatrix{B\lfb{8\alpha} && B_1},
$$
and
$$
\xymatrix{A_1\rd{\alpha_1} & B_1^{-1}}
$$
where $\alpha_1=2^{23}\alpha^9$. Applying Theorem~\ref{th-46} to $A_1$
and $B_1$, we get a $\beta$-approximate subgroup $H$ and a set $X$
with
$$
\xymatrix{H\lfb{8\alpha_1^{14}} && A_1\fb{1} & A}
$$
and
$$
\xymatrix{X\fb{\gamma} & 1},\quad \xymatrix{A_1\sub & XH},\quad
\xymatrix{B_1\sub & HX},
$$
where
$$
\beta=2^{21}\alpha_1^{80}=2^{1861}\alpha^{720},\quad\quad
\gamma=2^{28}\alpha_1^{104}=2^{2420}\alpha^{936},
$$
and moreover
$$
\xymatrix{\nfold{H}{3}\fb{\beta_3} & H}
$$
where $\beta_3=2^{10}\alpha_1^{40}=2^{930}\alpha^{360}$.
\par
Applying the pigeonhole principle, we find $x$ such that
$$
\xymatrix{A\lfb{8\sqrt{2}\alpha} && A_1\fb{\gamma} & A_1\cap xH\sub &
  A\cap xH}
$$
and $y$ with
$$
\xymatrix{B\lfb{8\alpha} && B_1\fb{\gamma} & B_1\cap Hy\sub & B\cap
  Hy}.
$$
\par
This gives what we want with
$$
\beta_1\leq 8\sqrt{2}\alpha\gamma
\leq 2^{2424}\alpha^{937},\quad\quad
\beta_2=8\alpha_1^{14}=2^{325}\alpha^{126}.
$$
\end{proof}


\end{document}